 \documentclass[11pt,letterpaper]{amsart}
\usepackage{graphicx,verbatim,amssymb,amsmath,multirow,bbold}
\usepackage[all]{xy}
\usepackage[active]{srcltx}
\usepackage{hyperref}

\DeclareSymbolFont{bbold}{U}{bbold}{m}{n}
\DeclareSymbolFontAlphabet{\mathbbold}{bbold}

\newtheorem{thm}{Theorem}[section]
\newtheorem{cor}[thm]{Corollary}
\newtheorem{lem}[thm]{Lemma}
\newtheorem{prop}[thm]{Proposition}

\newtheorem*{mainthm}{Main Theorem}

\theoremstyle{definition}
\newtheorem{dfn}[thm]{Definition}
\theoremstyle{remark}
\newtheorem{rem}[thm]{Remark}
\newtheorem{ex}[thm]{Example}
\numberwithin{equation}{section}

\def\C{\mathbb{C}}

\def\Q{\mathbb{Q}}
\def\R{\mathbb{R}}

\def\Z{\mathbb{Z}}
\def\1{\mathbbold{1}}

                \def\Fc{\mathcal{F}}
               
\def\Jc{\mathcal{J}}           \def\Kc{\mathcal{K}}        \def\Lc{\mathcal{L}}
             
\def\Pc{\mathcal{P}}

\def\Uc{\mathcal{U}}
\def\Vc{\mathcal{V}}

\def\Xc{\mathcal{X}}

\def\Bf{\mathfrak{B}}

\def\Sf{\mathfrak{S}}

\def\Gs{\mathsf{G}}
\def\Ts{\mathsf{T}}

\def\al{\alpha}
\def\be{\beta}

\def\eps{\varepsilon}
\def\la{\lambda}

\def\ta{\theta}

\def\d{\partial}
\def\0{\varnothing}
\def\sm{\setminus}
\def\ol{\overline}
\def\le{\leqslant}
\def\ge{\geqslant}
\def\<{\langle}
\def\>{\rangle}

\def\dom{\mathrm{dom}} 

\def\area{\mathrm{area}}
\def\Inv{\mathrm{Inv}}
\def\sign{\mathrm{sign}}
\def\rk{\mathrm{rk}}

\def\PET{\mathrm{PET}}
\def\IET{\mathrm{IET}}
\def\lev{\mathrm{lev}}
\def\cent{\mathrm{cent}}

\def\UT{\mathrm{UT}}
\def\uc{\mathbb{S}^1}

\begin{document}

\title{Aperiodic points for outer billiards}

\author[A. Belyi]{Anton Belyi}

\author[A. Kanel-Belov]{Alexei Kanel-Belov}

\author[Ph. Rukhovich]{Philipp Rukhovich}

\author[V.~Timorin]{Vladlen Timorin}

\address[Alexei Kanel-Belov]
{Bar--Ilan University, Building 216, 5290002 Ramat-Gan, Israel}

\address[Anton Belyi and Alexei Kanel-Belov and Philipp Rukhovich]
{Moscow Center for Advanced Studies, Kulakova str. 20, Moscow 123592, Russia}

\address[Vladlen~Timorin]
{Faculty of Mathematics, HSE University, 6 Usacheva ul., Moscow, 119048, Russia}

\email{vtimorin@hse.ru}

\thanks{The second named author has been partially supported by Israel Science Foundation [1623/16].
The fourth named author has been partially supported by the HSE University Basic Research Program.
}

\dedicatory{Dedicated to the memory of V.Z. Grines}

\begin{abstract}
Euclidean outer billiard on a regular polygon (that is not a triangle, square or a hexagon) has aperiodic points, i.e.,
 points where all iterates of the outer billiard map are defined and yield pairwise distinct images.
This result answers a question R. Schwartz posed at ICM 2022.
\end{abstract}

\maketitle

\section{Introduction}
Fix an orientation of the real affine plane $\R^2$.
Then, the two half-planes bounded by a given oriented line can be uniquely marked as
 the \emph{left} and the \emph{right} half-planes, that is,
 as lying on the left and, respectively, on the right of the line.
If $a$ and $b$ are two different points in the plane, then \emph{the line $ab$}
 will mean the line passing through $a$, $b$ and oriented from $a$ to $b$.
Say that a point lies \emph{on the left} (resp., \emph{on the right}) of the line $ab$
 if it lies in the left (resp., right) half-plane bounded by $ab$.
\begin{figure}
  \centering
  \includegraphics[height=.45\textwidth]{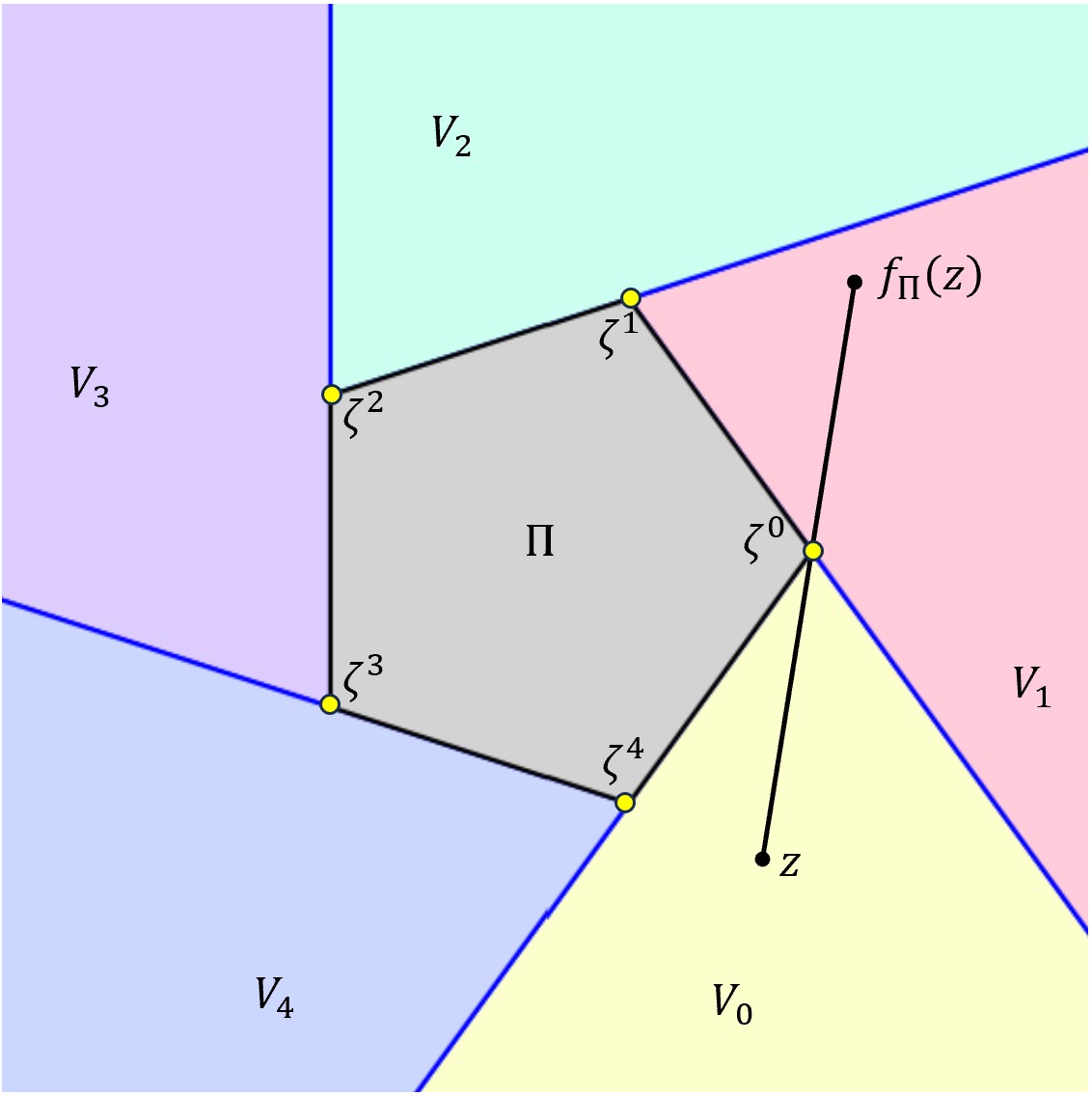}
  \includegraphics[height=.45\textwidth]{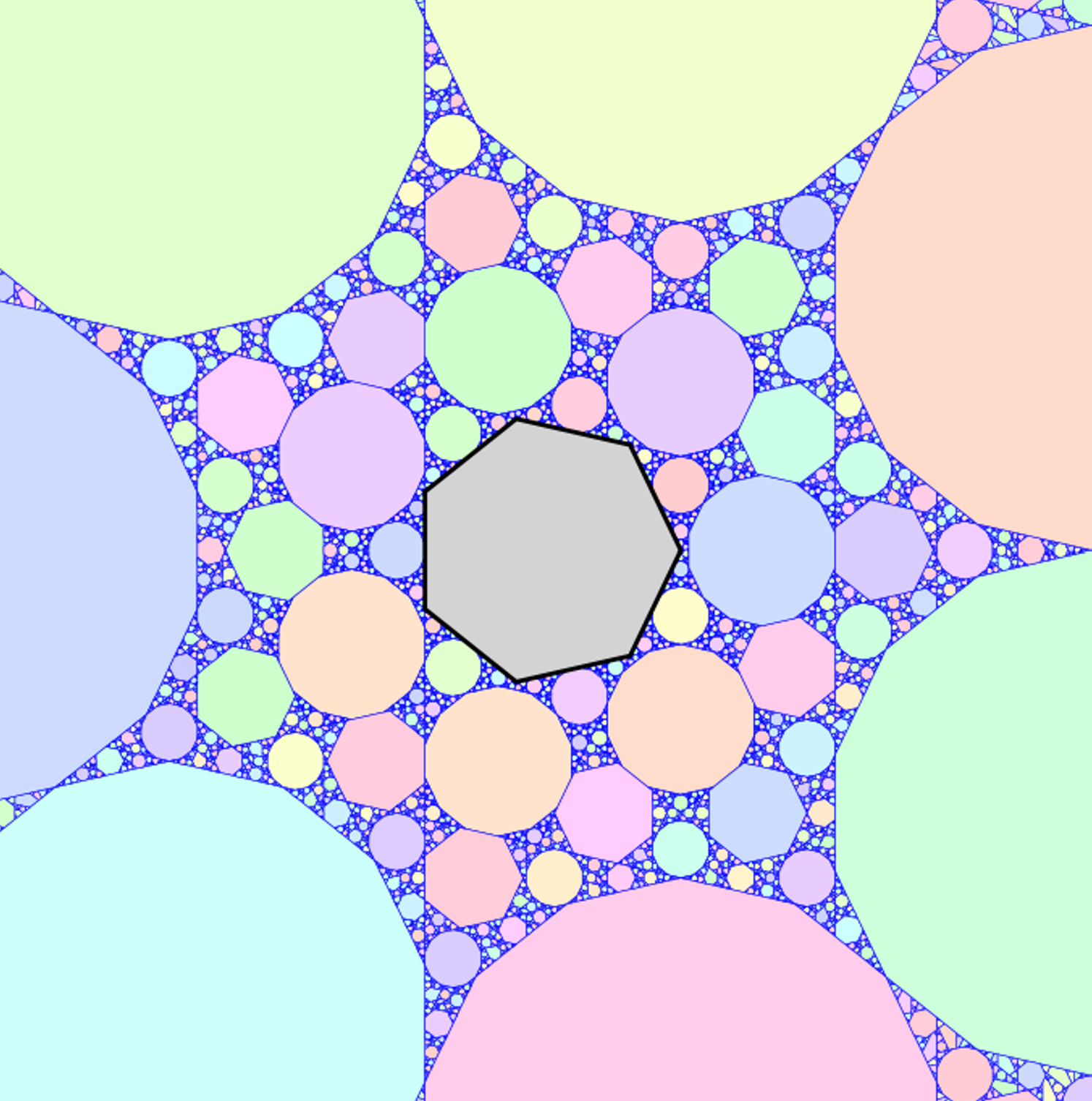}
  \caption{\small Left: outer billiard map $f=f_\Pi$ on a regular pentagon $\Pi$.
  Here, the vertices of $\Pi$ are the roots of unity $\zeta^k$, where $\zeta=\exp(2\pi \mathbf{i}/5)$.
  The domain $\dom(f)$ of $f$ is the union of the 5 sectors $V_k$,
   and the map $f_\Pi$ acts on $V_k$ as the half-turn about $\zeta^k$.
   Right: some periodic domains for the outer billiard on a regular 7-gon.}\label{fig:dual}
\end{figure}
Consider an open convex polygon $\Pi$ in $\R^2$.
The \emph{right (resp., left) outer billiard} on $\Pi$
 is the map $f_\Pi$ defined for points $z\in\R^2$
 such that, for some vertex $a$ of $\Pi$, all other vertices of $\Pi$ lie on the left
 (resp., right) of $za$.
For such points $z$, one sets $f_\Pi(z)=2a-z$, that is,
 $f_\Pi$ acts on $z$ as the central reflection (=a half-turn) about the vertex $a$.
It is clear that $f_\Pi$ is defined on $\R^2\sm\ol\Pi$, except on
 finitely many rays, each emanating from a vertex of $\Pi$
 and lying on a support line of $\Pi$.
Here, $\ol\Pi$ denotes the closure of $\Pi$ in $\R^2$.
We are concerned with outer billiards on regular $N$-gons with integer $N>2$.
In the sequel, $f_\Pi$ will stand for the right outer billiard on $\Pi$.
See Fig. \ref{fig:dual} for an illustration.

Recall that a point $z\in\R^2$ is called \emph{aperiodic} for $f_\Pi$ if all iterates $f^n_\Pi(z)$
 are defined for $z$ (where $n$ runs through all integers, positive or negative) and are all pairwise distinct.
Our main result is the following.

\begin{mainthm}
Outer billiard on any regular $N$-gon with $N>4$ and $N\ne 6$ has aperiodic points.
\end{mainthm}

The Main Theorem is proved by rather general methods,
 which may be applicable to wider classes of piecewise isometries,
 and possibly beyond the piecewise linear realm.

\subsection{Overview and related works}
\label{ss:overview}
More generally, outer billiard can be defined for any compact convex figure in the plane
 (see \cite{Tab93,DT05,tab05} for overviews).
First definition of outer billiard maps may be due to M. Day \cite{Day47};
 it reappeared in B. Neumann \cite{Neu59}.
J. Moser \cite{Mos78} suggested outer billiard as a crude model of planetary motion: a simplification,
 in which the stability problem still has subtleties of the original context.
Dynamical stability of outer billiard means that all orbits are bounded.
This is the case when the figure is strictly convex and has sufficiently smooth boundary \cite{mos01,dou82}.
It is also the case \cite{VSh87,Kol89,GS92} for the so called \emph{quasi-rational polygons}, which include all
 lattice polygons as well as all regular polygons.
See also the thesis of D. Genin \cite{Gen05} for the boundedness of orbits in a family of trapezoidal outer billiards.
However, as was discovered in \cite{Sch07}, outer billiards on convex quadrilaterals
 can have unbounded orbits (\cite{Sch09,Sch19} study these orbits in more detail).

Polygonal outer billiards are arguably the principal examples of Euclidean piecewise rotations
 (see \cite{Goe00,Goe03} for overviews of piecewise isometries, including piecewise rotations).
Piecewise rotations are a natural generalization of interval exchange transformations.
They also found applications in electrical engineering, in particular, in the study of digital filters
 \cite{ACP97,CL90,Dav92,KWC96,Ogo92}.
In \cite{YFW10}, a sufficient condition for the existence of periodic points is given for a general class of
 piecewise isometries.

Previous results concerning outer billiards on regular polygons are as follows.
Apart from “trivial” (or “integrable”) cases of regular $N$-gons for $N=3,4,6$
 (these ``trivial'' cases are characterised by the property that the group of isometries generated by
 the half-turns about all the vertices acts properly discontinuously on the plane),
 all understood cases
 possess dynamical \emph{self-similarities},
 i.e., Euclidean similarities that conjugate the first return map of some polygon with the first return map of a strictly smaller polygon.
S. Tabachnikov \cite{Tab95} described a self-similarity for $N=5$ and its implications
 for orbit coding and the Hausdorff dimension of the Julia set,
 which eventually led to a precise symbolic model in this case
 (Bedaride-Cassaigne \cite{BC11}).
G.H. Hughes \cite{Hug12,Hug13,Hug16,Hug21} has a number of computer assisted experimental conjectures
 concerning the values of $N$ up to about 20.
Cases $N=8, 10, 12$ are rigorously covered by
 \cite{BC11,Ruk22}.
Previous results of R. Schwartz \cite{Sch14} embedded the case $N=8$ into a continuous family with
 an action of a suitably defined renormalization operator.
Yet earlier, in \cite{AKT01}, a similar piecewise rotation of the torus was studied.
Cases $N=7$, $9$ display multi-fractal structure (independent self-similarities leading to uncountable families of
 pairwise disjoint closed invariant subsets of aperiodic points), which may be suggested by experimental data of \cite{Hug13};
 a computer aided proof of these phenomena is a work in progress by V. Zgursky and the authors of this paper.

In the survey article \cite{Sch21} based on his ICM 2022 address, Schwartz notes:
\begin{enumerate}
\item[]
``\emph{The cases $n = 8$, $10$, $12$ also have a self-similar structure. Without having a reference,
I have the sense that the case $n = 7$ is somewhat understood in the sense that there are some
regions of renormalization. I think that the cases $n = 9$, $11$ are not understood at all. G.
Hughes [63] has made beautiful and detailed pictures of outer billiards on regular polygons.
These pictures (and earlier ones) suggest }

\textbf{Conjecture 6.5.} Outer billiards on the regular $n$-gon has an aperiodic orbit if $n \ne 3$, $4$, $6$.

\emph{I think that this is not known aside from $n = 5$, $8$, $10$, $12$, and perhaps $n = 7$.}''
\end{enumerate}
This paper settles the conjecture.
Previously, E. Gutkin and N. Simanyi \cite{GS92} posed a more general question:
``\emph{Let $\Pi$ be a quasi-rational but not rational polygon.
 We expect that the outer billiard about $\Pi$ to have nonperiodic orbits.}''
 Methods of this paper may also shed some light on this question but this is a plan for the future.

\subsection{Methods}
\label{ss:method}
Even though the outer billiard map itself is an example of a piecewise rotation,
 the domain of this piecewise rotation is unbounded, which makes things more complicated.
A remedy to this complication has been known for a long time:
 there are bounded regions invariant under $f_\Pi$;
 the restrictions of $f_\Pi$ to these regions
 are bounded piecewise rotations, some of which are easy to describe.
Dynamics of $f_\Pi$ on a bounded invariant polygon can be easily
 reduced to that of a \emph{polygon exchange transformation},
 i.e., an invertible piecewise isometry acting by parallel translations on all components of its domain.
Next, the main issue when looking for aperiodic points of
 a polygon exchange transformation of a bounded polygon is
 to rule out full periodicity of the map on the complement of finitely many lines.

For this problem --- and here comes the principal novelty of this paper --- we employ
 \emph{scissors congruence invariants}.
Two polygons are \emph{scissors congruent} if one can cut one polygon
 into finitely many polygonal pieces and assemble the other polygon from these pieces.
If $\Gs$ is a subgroup of 
 isometries of $\R^2$,
 whose elements are the only ones allowed to act on the pieces, then we speak
 of $\Gs$-scissors congruence.
A case of special importance for us is $\Gs=\Ts$: the group of all translations.
By ($\Gs$-)\emph{scissors congruence invariants}, we mean functions on
 polygons that take the same values on $\Gs$-scissors congruent polygons.
Dynamical versions of $\Ts$-scissors congruence invariants are
 conjugacy invariants of polygon exchange transformations.\footnote{Polygon
 exchange transformations are themselves scissors congruences,
 and dynamic invariants are, in a sense, scissors congruence invariants of these
 scissors congruences.}

Such invariants, in principle, can lead to very general necessary conditions of periodicity,
 in the form of vanishing of these invariants.
A sample condition of this type follows from Theorem \ref{t:additiv}:
 for a periodic polygon exchange transformation, any dynamic invariant
 (as described in Section \ref{ss:val}) must vanish.
The principle idea of the Main Theorem is simple: compute suitable
 dynamic invariants and show they are non-zero.
Once it is established that a bounded restriction of $f_\Pi$ is not periodic,
 establishing the existence of aperiodic points is not immediate,
 but this part is less conceptual and more technical (Sections \ref{s:flag} and \ref{s:red1D}).

Using earlier results of \cite{Ruk22} (see also an informal overview in Section \ref{ss:Rukh}),
 the Main Theorem can be reduced to the case of even $N$:
 for an odd $N$, there is a certain correspondence between outer billiards on the $N$-gon
 and on the $2N$-gon, allowing to conclude that the existence of aperiodic points in one
 of these systems implies the same for the other.
Thus, it suffices to assume that $N$ is even.
We also need to reduce the dynamics of $f_\Pi$ to that of a polygon exchange transformation,
 i.e., a piecewise isometry acting by translations.
One obvious choice is the second iterate $f^{2}_\Pi$, however, a different choice allows to significantly simplify the computations.
Namely, consider the half-turn $M$ about the center of $\Pi$, and form the composition $g_\Pi:=M\circ f_\Pi$.
It acts as a piecewise translation, see Fig. \ref{fig:gX}.
Also, an $f_\Pi$-invariant bounded domain is invariant under $M$ and hence under $g_\Pi$,
 and the dynamics of $g_\Pi$ is not much different from that of $f_\Pi$, since $f_\Pi$ commutes with $M$.
For these reasons, if $g_\Pi$ is not periodic, then neither is a bounded restriction of $f_\Pi$.
To establish non-periodicity of $g_\Pi$,
 we use a two-dimensional dynamic invariant based on the Hadwiger--Glur valuation \cite{HG51}.
Explicit computation of this invariant and verification that the result is nonzero are not hard.

\subsection{Organization of the paper}
Section \ref{s:PWI} discusses general concepts associated with scissors congruences,
 including their dynamical aspects.
As a special case, we consider polygon exchange transformations,
 which act as piecewise translations.
\emph{Dynamic invariants} of polygon exchange transformations
 of a given polygon
 are certain homomorphisms of the group of such transformations to
 the additive group of a real vector space.
It follows that these dynamic invariants are conjugacy invariants under
 piecewise translational conjugacies.
\emph{Dynamic Hadwiger invariants} can be obtained as special
 cases of more general dynamic invariants; they are based
 on (translation) scissors congruence invariants of Hadwiger and Glur \cite{HG51}
 initially discovered as a means of solving translation scissors congruence
 problem for planar polygons.

Assuming that $N$ is even, Section \ref{s:bnded} deals with the outer billiard map $f_\Pi$
 on a regular $N$-gon $\Pi$.
There is a bounded domain $\Upsilon$ invariant under $f_\Pi$; see Section \ref{ss:neck}.
As explained above, the map $g_\Upsilon:=M\circ f_\Pi|_\Upsilon$ is a polygon
 exchange transformation of this bounded domain.
Here $M$ is the half-turn about the center of $\Pi$.
Computing the dynamic Hadwiger invariant of $g_\Upsilon$ (Theorem \ref{t:Inv-gX}) and establishing that it is nonzero
 yields that $f_\Pi$ is not periodic on $\Upsilon$ (Theorem \ref{t:fX-nonper}).
We still need to upgrade this result to existence of aperiodic points, which is done in Sections \ref{s:flag} -- \ref{s:red1D}.
By way of contradiction, we assume that $f_\Pi$ has no aperiodic points.

Section \ref{s:flag} describes 
 the action of $f_\Pi$ and more general polygon exchange transformations on spaces of flags.
General facts from topological dynamics are used
 to extract an invariant subset $\Kc$ of the flag space that has no wandering points and
 whose projection to the plane is contained in a finite union of line segments consisting of boundary points.
Note that $\Kc$ appears in an argument by contradiction (assuming $f_\Pi$ has no aperiodic points);
 for this reason, we cannot give specific examples of such $\Kc$s for outer billiards.
The Main Theorem reduces to certain properties of $\Kc$, see Theorem \ref{t:rem-core},
 and the latter are discussed in Section \ref{s:red1D}.
Namely, the dynamics on $\Kc$ is basically that of an interval exchange transformation,
 which yields an invariant subpolygon with zero dynamic invariants.
Removing this subpolygon allows to perform an induction step and, ultimately, arrive at a contradiction.

Possible directions for future research are discussed in Section \ref{s:sum}.

\subsection{Acknowledgements}
Valery Zgursky was a member of the team at an earlier stage of the project
 and co-authored a preliminary version of the paper.
We are grateful to Victor Kleptsyn and Alexey Semenov for fruitful discussions,
 and to Alexandra Skripchenko for consultations on the topic of interval exchange transformations.
Finally, we thank the referees for careful reading and useful suggestions that
 greatly improved the exposition (one informal explanation suggested by a referee is included
 almost verbatim in the current text; the reader will recognize it by the more idiomatic English).

\section{Dynamics of scissors congruences}
\label{s:PWI}
In this section, we introduce terminology related to scissors congruences and, specifically,
 polygon exchange transformations, which are two dimensional generalizations of interval exchange maps.

\subsection{Basic terminology}
\label{ss:PR}
Given a compact polygon $P$ in the plane\footnote{Here, $P$, $Q$, etc., denote generic polygons,
 while $\Pi$ is reserved for the outer billiard table, which is almost always assumed to be
 a regular $N$-gon.}, not necessarily convex, even not necessarily connected,
 one can cut it into finitely many polygonal pieces and form a new polygon $Q$ from these pieces.
In this case, we talk about a \emph{scissors congruence} between $P$ and $Q$
 and call the polygons $P$ and $Q$ \emph{scissors congruent}.
A more precise definition specifies a precise meaning of ``cutting into pieces''
 as well as describes what kind of motions are allowed on the pieces.

\begin{dfn}[Scissors congruence]
\label{d:scicong}
Fix a subgroup $\Gs$ in the full group of Euclidean isometries of the plane.
Also, consider two open bounded polygons $P$ and $Q$ in the plane, not necessarily convex
 and not necessarily connected.
Suppose that
$$
\ol P=\ol{P_1\cup\dots\cup P_n},\quad \ol Q=\ol{Q_1\cup\dots \cup Q_n},
$$
 where $P_i$ and $Q_i$ are open convex polygons such that $g_i P_i=Q_i$ for some $g_i\in\Gs$,
 moreover, $P_i\cap P_j=Q_i\cap Q_j=\0$ for $i\ne j$.
In this case, $P$ and $Q$ are said to be \emph{$\Gs$-scissors congruent}.
The map $f$ defined on $P_1\cup\dots\cup P_n$ and acting as $g_i$ on each $P_i$
 is called the corresponding ($\Gs$-)\emph{scissors congruence}.
\end{dfn}

A scissors congruence $f$ between $P$ and $Q$ is therefore a partially defined map;
 write $\dom(f)$ for the \emph{domain} of $f$, that is, for the union of $P_i$s.
Clearly, the inverse map $f^{-1}$ is defined on $Q_1\cup\dots\cup Q_n$; it is a
 scissors congruence between $Q$ and $P$.
The domain of $f^{-1}$ coincides with the image of $f$, that is $\dom(f^{-1})=f(\dom(f))$.
Now assume that $P=Q$; in this case, $f$ generates a dynamical system.
Given a positive integer $n$, write $f^n$ for the $n$-th iterate of $f$ and $f^{-n}$ for the $n$-th iterate of $f^{-1}$.
Setting also $f^0=id$ makes iterates $f^n$ well defined scissors congruences for all integer values of $n$.
Note that the domain of $f^{n+1}$, for $n\ge 0$,
 can be obtained as $f^{-n}(\dom(f))$; usually, it is strictly smaller than $\dom(f^n)$ and
 is obtained from $\dom(f^n)$ by removing several straight line segments
 (those mapping to $P\sm\dom(f)$ under $f^n$).

\emph{Scissors automorphisms} of a bounded polygon $P$ are defined
 as scissors congruences between $P$ and itself; two such congruences
 $f:\dom(f)\to P$ and $g:\dom(g)\to P$ are identified if $f=g$ on $\dom(f)\cap\dom(g)$.
With this identification, all scissors automorphisms of $P$ form a group under composition.
They belong to a more general class of \emph{piecewise isometries}
 discussed, e.g., in \cite{Goe00}.
We adopt standard terminology and notation related to partially defined maps.
In particular, for a partially defined map $f$ of a set $P$ and a subset $A\subset P$,
 the \emph{image} $f(A)$ is defined as $f(A\cap\dom(f))$.
Given a scissors congruence of a polygon $P$ and a subpolygon $A\subset P$,
 say that $A$ is \emph{invariant} under $f$ if $f(A)\subset \ol A$;
 then automatically $f^{-1}(A)\subset \ol A$ since the area is preserved.

\begin{rem}[Connection with outer billiards]
  \label{r:outer}
Consider the outer billiard map $f_\Pi$ associated with a convex polygon $\Pi\subset\R^2$.
We claim that $f_\Pi$ is a \emph{piecewise rotation} of $\R^2\sm\ol\Pi$, that is,
 $f$ acts on each component of $\R^2\sm\ol\Pi$ as an orientation preserving
 Euclidean isometry.
Indeed, label the vertices of $\Pi$ as $\zeta^0$, $\dots$, $\zeta^{N-1}$, and let $V_i$
 be the locus of all points $z\in\R^2\sm\ol\Pi$ such that $f_\Pi$ acts on $z$ as the reflection in $\zeta^i$.
Then $V_i$ are convex cones (=sectors) bounded by continuations of the edges of $\Pi$.
On each $V_i$, the outer billiard map acts as the rotation by $\pi$, that is, as a half-turn, with center $\zeta^i$.
Note that the domain of $f$ is the union of all $V_i$, that is,
 the complement in $\R^2\sm \ol\Pi$ of $N$ straight line rays emanating from the vertices of $\Pi$.
See Fig. \ref{fig:dual} for the case when $\Pi$ is a regular polygon.
A way of obtaining a scissors congruence out of $f_\Pi$ is restricting the map $f_\Pi$
 to a suitably chosen invariant bounded polygon; this is indeed possible, as we will see later.
\end{rem}

\begin{dfn}[Boundary, periodic, aperiodic points]
\label{d:aper}
Given a scissors automorphism $f$ of $\Delta$,
 define a \emph{boundary point} of $f$ as a point $z\in\ol\Delta$, for which
there is $n\in\Z$ such that $z\notin \dom(f^n)$, that is, some positive or negative iterate of $f$
 is not defined at $z$ (which includes the case $z\in\d\Delta$).
If $f^n(z)=z$ for some $n\ne 0$ (which means, in particular, that $z\in\dom(f^n)$), then the
 point $z$ is called \emph{periodic} under $f$.
The smallest positive integer $n$ with this property is called the \emph{period} of $z$.
It is immediate from the definition that periodic points are not boundary points.
\emph{Aperiodic} points are defined as points of $\Delta$ that are neither periodic points nor boundary points.
\end{dfn}

See \cite{Goe03} for an overview of some results and open problems
 in dynamics of piecewise isometries. 

\begin{dfn}[Periodicity]
  \label{d:per}
  A scissors automorphism $f$ of $\Delta$ is called \emph{periodic} if $f^n=id_{\dom(f^n)}$ for some integer $n>0$.
The smallest such $n$ is called the \emph{period} of $f$.
Equivalently, $f$ is periodic of period $n$ if it represents an order $n$ element
 in the group of scissors automorphisms of $\Delta$.
\end{dfn}

Let $\Ts$ be the group of all translations of the plane; $\Ts$-scissors automorphisms
 are also known as  \emph{translation scissors congruences} or as \emph{polygon exchange transformations}.
The group of all polygon exchange transformations of a polygon $\Delta$ is denoted by $\PET(\Delta)$.
Polygon exchange transformations admit a straightforward generalization
 to arbitrary dimensions, where one considers \emph{polytope exchange transformations}.
Dimension one polytope exchange transformations are called \emph{interval exchange transformations};
 the group of all interval exchange transformations of an interval $I$ is denoted by $\IET(I)$.
Periodic polygon exchange transformations yield, in particular, periodic elements of the group $\PET(\Delta)$.

\subsection{Sah--Arnoux--Fathi invariants}
\label{ss:SAF}
We proceed by reviewing an invariant of an interval exchange transformation called
 the \emph{Sah--Arnoux--Fathi invariant} (abbreviated as \emph{SAF invariant}) \cite{Arn81,Sah81}.
Consider an interval exchange transformation $g$ defined on an interval $I$.
Then the SAF invariant $\Inv(g)$ is an element of $\R\otimes_\Q\R$ defined as follows.
If subintervals $I_1$, $\dots$, $I_n$ of lengths $\lambda_1$, $\dots$, $\lambda_n$, respectively,
 form a partition of $I$, and the restriction of $g$ to each $I_i$ is a translation by $\tau_i\in\R$, then
$$
\Inv(g)=\lambda_1\otimes \tau_1+\lambda_2\otimes \tau_2+\dots + \lambda_n\otimes \tau_n.
$$

The following is a summary of some (classical) important properties of the SAF invariant
 (a nice and comprehensive exposition of these results is contained in \cite{Vor17}).
Every periodic interval exchange transformation has zero SAF invariant.
Also, every interval swap map --- an interval exchange transformation that
 interchanges two intervals with disjoint interiors and otherwise acts as the identity --- has zero SAF invariant.
The map $g\mapsto \Inv(g)$ is a homomorphism from the group of all interval exchange
 transformations of a given interval to the additive group of $\R\otimes_\Q\R$;
 the image of this homomorphism coincides with $\R\wedge_\Q\R$.

The map $\Inv$ being a group homomorphism means that
$$
\Inv(f\circ g)=\Inv(f)+\Inv(g)
$$
for any two interval exchange transformations $f$, $g$ of the same interval.
It follows immediately that $\Inv(f^n)=n\Inv(f)$, hence $f$ being periodic implies that $\Inv(f)=0$.
Also, it follows that $\Inv(g\circ f\circ g^{-1})=\Inv(f)$, that is, $\Inv$ is a conjugacy invariant.

Recall that $\R\otimes_\Q\R$ is a vector space over $\R$ whose basis over $\R$ can
 be identified with a basis of $\R$ over $\Q$.
More precisely: choosing any basis $\Bf$ of $\R$ over $\Q$ yields an identification
 of $\R\otimes_\Q \R$ with the real vector space freely generated by $\Bf$.
Thus, $\R\otimes_\Q \R$ and $\R\wedge_\Q\R$ are infinite dimensional real vector spaces.
Any invariant with values in these spaces can be alternatively viewed as
 an infinite collection of real-valued invariants.
For example, choosing any $\Q$-bilinear map $\beta:\R\times\R\to\R$ yields the
 corresponding invariant
$$
\Inv_\beta(g):=\beta(\lambda_1,\tau_1)+\beta(\lambda_2,\tau_2)+\dots +\beta(\lambda_n,\tau_n).
$$

It is straightforward to define interval exchange transformations of any metric oriented graph
 (a graph whose edges carry a Euclidean metric and an orientation).
Also, since changing an interval exchange transformation at finitely many points
 does not change the corresponding element of the group $\mathrm{IET}(I)$,
 it is also irrelevant what the domain of a given interval exchange map looks like as long as
 it is a finite metric graph.
For this reason, we will speak freely of interval exchange maps defined on various graphs.

Next, we want to generalize the SAF invariant to higher dimensional
 polytope exchange transformations.
Specifically, define a \emph{dynamic invariant} as a group
 homomorphism from $\mathrm{PET}(P)$ to $\R$.
In what follows, new dynamic invariants will be introduced in 2D (they also make sense in dimensions $>2$).

\subsection{Valuations}
\label{ss:val}
We first recall some classical notions and results related to valuations in a non-dynamical setting.
The following definition is standard.

\begin{dfn}[Valuations]
\label{d:val}
Let $\Pc$ be the set of all compact convex polygons in the plane, including $\0$.
A \emph{valuation} on convex polygons with values in $\R$ is by definition a function
 $\mu:\Pc\to\R$ such that $\mu(\0)=0$ and
$$
\mu(P\cup Q)=\mu(P)+\mu(Q)-\mu(P\cap Q)
$$
 for convex polygons $P$, $Q$ such that $P\cup Q$ is also convex.
This property of valuations is called \emph{finite additivity}.
A valuation is therefore the same as a finitely additive (signed) measure.
\end{dfn}

It is straightforward to check that any valuation $\mu$ extends to
 arbitrary, not necessarily convex, bounded polygons,
 and satisfies the same finite additivity property on this larger class of sets.

A valuation $\mu$ is said to be \emph{simple} if $\mu(P)=0$ whenever $\dim(P)<2$.
Say that $\mu$ is \emph{translation invariant} if $\mu(P+\vec v)=\mu(P)$ for all
 convex polygons $P$ and all $\vec v\in\R^d$.
Here, $P+\vec v$ means the result of the translation of $P$ by the vector $\vec v$.
More generally, one can speak of \emph{$\Gs$-invariant valuations}, where $\Gs$
 is any group of affine transformations of $\R^d$.
Simple $\Gs$-invariant valuations are useful for solving $\Gs$-scissors congruence problems.
Namely, if $P$ and $Q$ are $\Gs$-scissors congruent, then any simple $\Gs$-invariant valuation
 takes the same value on $P$ and $Q$.
In other words, simple invariant valuations help distinguish pairs of
 polytopes that are \emph{not} scissors congruent.

\begin{ex}[Area]
  Fix an area form on $\R^2$.
Then one can assign the (unsigned) \emph{area}, denoted $\area(P)$, to any polygon $P$.
Area is a simple valuation invariant under the full group of isometries of $\R^2$.
Now let $\xi:\R\to\R$ be any $\Q$-linear function.
Then $\xi\circ\area$ is also a simple isometry invariant valuation;
 this construction gives uncountably many such valuations, linearly independent over $\R$.
\emph{Perimeter} (suitably extended to degenerate polygons) is also an isometry invariant
 valuation, however, it is not simple.
\end{ex}

\begin{ex}[Hadwiger--Glur valuations]
Fix a line $L\subset\R^2$ together with a \emph{co-orientation}, that is,
 a labeling of the two components of $\R^2\sm L$ as
 \emph{the positive and the negative sides} of $L$.
Given a co-orientation of $L$, we can consistently define the positive and the negative
 sides of any line parallel to $L$.
Let now $P$ be a bounded polygon, not necessarily convex, and define $\mu(P)$ as follows.
Every edge $e$ of $P$ parallel to $L$ contributes an additive term $\pm |e|$ to $\mu(P)$,
 where $|e|$ is the length of $e$, and the sign is plus or minus
 depending on whether $P$ is on the positive or on the negative side of $e$ (locally near $e$,
 with respect to the co-orientation of the line containing $e$ that is consistent
 with the co-orientation of $L$).
The sum $\mu(P)$ of such terms is called the \emph{Hadwiger invariant} of $P$,
 and the valuation $\mu:P\mapsto \mu(P)$ is called the \emph{Hadwiger--Glur valuation}.
It was introduced by Hadwiger and Glur in \cite{HG51}, see also \cite{Had13}.
Observe that $\mu$ is simple and translation invariant.
Two polygons are $\Ts$-scissors congruent if and only
 if they have the same area and the same value of the Hadwiger invariants for all $L$, cf. \cite{HG51}.
Even though there are uncountably many Hadwiger--Glur valuations,
 only finitely many of them take nonzero values at a given polygon.
\end{ex}

Let $P$ be a bounded polygon.
By a \emph{dynamic invariant}, we mean a group homomorphism $\Psi$ from $\mathrm{PET}(P)$
 to the additive group of some $\Q$-vector space.
For any two transformations $f$, $g\in\mathrm{PET}(P)$, we therefore have
$$
\Psi(f\circ g)=\Psi(f)+\Psi(g).
$$
It follows that $\Psi(f^n)=n\Psi(f)$ for every $n\in\Z$.
Therefore, if $f$ is periodic, then \emph{any} dynamic invariant must vanish on $f$.
Dynamic invariants are, indeed, congugacy invariants ($\Psi(g\circ f\circ g^{-1})=\Psi(f)$).
Below, we give a general construction of two-dimensional dynamic invariants.

Suppose that $\mu$ is a simple translation invariant valuation on polygons in $\R^2$.
Take any $f\in\mathrm{PET}(P)$, whose domain consists of open convex polygons $P_i$
 (where $i$ runs from $1$ to $n$)
 and whose restriction to $P_i$ is the translation by a vector $\vec v_i$.
Define
$$
\Inv_{\mu}(f)=\sum_{i=1}^{n} \mu(P_i)\otimes \vec v_i\in \R\otimes_\Q\R^2
$$
 and call it the \emph{$\mu$-invariant} of $f$.
The main property of this invariant is the following.

\begin{thm}
  \label{t:additiv}
The function $\Inv_{\mu}$ is a group homomorphism:
$$
\Inv_{\mu}(g\circ f)=\Inv_{\mu}(g)+\Inv_{\mu}(f).
$$
In particular, any $\Q$-bilinear functional $\be:\R\times \R^2\to\C$ defines a dynamic invariant
 $\Inv_{\mu,\be}$ with values in $\C$, which can be computed as the sum of $\beta(\mu(P_i),\vec v_i)$ for $i=1$, $\dots$, $n$.
\end{thm}

\begin{proof}
Suppose that the domain of $f$ has components $P_1$, $\dots$, $P_n$, and the domain $g$ has components $Q_1$, $\dots$, $Q_m$.
Write $f_i:=f|_{P_i}$ and $g_j:=g|_{Q_j}$.
Note that the components of $\dom(g\circ f)$ are those of the sets $P_{ij}=f_i^{-1}(Q_j)$ which are nonempty.
Let $\vec v_i$, $\vec w_j$ stand for the translation vectors of $f_i$, $g_j$, respectively,
 and $\Sf$ for the set of pairs $(i,j)$ such that $P_{ij}\ne\0$.
Also, for every $j=1$, $\dots$, $m$, let $\Sf_j$ stand for the set of $i$s with $(i,j)\in \Sf$,
 and write $\Sf^i$ for the set of all $j$ with $(i,j)\in \Sf$.
Then, on $P_{ij}$, the composition $g\circ f$ acts as a translation by vector $\vec v_i+\vec w_j$.
Computing $\Inv_{\mu}(g\circ f)$ by definition yields
$$
\sum_{(i,j)\in \Sf} \mu(P_{ij})\otimes (\vec v_i+\vec w_j)=
\sum_{i=1}^n \sum_{j\in \Sf^i} \mu(P_{ij})\otimes \vec v_i + \sum_{j=1}^m \sum_{i\in \Sf_j}\mu(P_{ij})\otimes \vec w_j.
$$
Fix $j$ and consider each term of the right hand side for this fixed value of $j$.
Note that the sum of $\mu(P_{ij})$ over all $i\in \Sf_j$ equals $\mu(Q_j)$,
 by translation invariance and simplicity of the valuation $\mu$.
It follows that the second term in the right hand side equals the sum of $\mu(Q_j)\otimes \vec w_j$ over all $j$,
 which is $\Inv_{\mu}(g)$ by definition of ${\mu}$-invariant.
Similarly, the sum of $\mu(P_{ij})$ over all $j\in \Sf^i$ equals $\mu(P_i)$, by translation invariance and simplicity of $\mu$.
The first term in the right hand side is then the sum over all $i$ from $1$ to $n$ of $\mu(P_i)\otimes \vec v_i$, which
 is $\Inv_{\mu}(f)$ by definition of ${\mu}$-invariant.
The theorem follows.
\end{proof}

The following corollary of Theorem \ref{t:additiv} will be used in the
 proof of the Main Theorem.

\begin{cor}
\label{c:Inv-noper}
Let $\Delta$ be a bounded polygon, and take $f\in\mathrm{PET}(\Delta)$.
If $\Inv_\mu(f)\ne 0$ for some simple $\Ts$-invariant valuation $\mu$,
 then $f$ has infinite order in the group $\mathrm{PET}(\Delta)$.
\end{cor}

\begin{proof}
Similarly to the case of interval exchange maps, the relation $f^n=id$ in the group $\mathrm{PET}(\Delta)$
 implies $n\,\Inv_\mu(f)=0$ by Theorem \ref{t:additiv}, hence $\Inv_\mu(f)=0$, a contradiction.
\end{proof}

Next, we discuss dynamical version of the Hadwiger--Glur valuation.

\begin{dfn}[Dynamic Hadwiger invariant]
\label{d:Inv_L}
Fix a co-oriented line $L$ in $\R^2$, and
 define the \emph{dynamic Hadwiger invariant} $\Inv_{L}$ associated with $L$ as $\Inv_{\mu_L}$, where
 $\mu_L$ is the Hadwiger--Glur valuation corresponding to $L$.
Recall that $\Inv_{L}$ takes values in $\R\otimes_\Q\R^2$.
\end{dfn}

For a given polygon exchange transformation $f$,
 the dynamic Hadwiger invariant $\Inv_L(f)$ associated with $L$ can be computed as follows.
A component $Q$ of $\dom(f)$ and an edge $E$ of $Q$ contribute the additive
 term $[E|Q]:=\sign[E|Q] (|E|\otimes \vec v)$ to the invariant, where $E$ is an edge of $Q$
 parallel to $L$ (we write $|E|$ for the length of $E$), and $\vec v$ is the translation vector of $f|_Q$.
The \emph{sign} $\sign[E|Q]$ is defined as $+1$ if $Q$ lies on the positive side of $E$
 and as $-1$ otherwise.
A pair $(Q,E)$ like that, i.e., with $Q$ being a component of $\dom(f)$ and $E$ being an edge of $Q$
 parallel to $L$, is called an \emph{$L$-adapted pair}.
It is said to be \emph{positive} or \emph{negative} according to whether $\sign[E|Q]$ is positive or negative.
The following lemma summarizes the above discussion.

\begin{lem}
  \label{l:adapt}
If $f$ is a polygon exchange transformation,
  then $\Inv_{L}(f)$ can be computed as the sum of $[E|Q]$ over all $L$-adapted pairs $(Q,E)$.
\end{lem}

\section{Bounded restrictions of outer billiards}
\label{s:bnded}
As in the introduction, we write $\Pi=\Pi_N$ for the regular $N$-gon inscribed in the unit circle such that
 one of its vertices lies in the positive $x$-semiaxis.
Under a suitable identification of $\R^2$ with $\C$ (so that the origin
 of $\R^2$ identifies with the zero complex number, and the positive semiaxis
 identifies with the set of positive real numbers), the vertices of $\Pi$ are
 precisely all $N$-th roots of unity; they have the form $\zeta^k$ for $\zeta:=\exp(2\pi \mathbf{i}/N)$
 and $k=0$, $\dots$, $N-1$.
Recall a general result of Rukhovich \cite{Ruk22}
 (see also the ``Twice-odd Lemma'' of \cite{Hug13} for an earlier experimental evidence).

\begin{thm}[\cite{Ruk22}]
\label{t:Rukh}
  Let $N>3$ be odd.
The map $f_{\Pi_N}$ has aperiodic points if and only if $f_{\Pi_{2N}}$ does.
\end{thm}

See Section \ref{ss:Rukh} for an informal explanation of why this theorem holds.
Thus, it suffices to prove the Main Theorem only for even values of $N$.
For this reason, we assume from now on and throughout the rest of the paper that $N$ is even.

\subsection{Vassal polygons}
\label{ss:neck}
Section \ref{ss:neck} claims no originality.
In ``setting the stage'' below, we essentially follow the construction of \cite{GS92,Tab93,BC11,Ruk22},
 however, the notation and some conventions will be different.
Let $R$ be the counterclockwise rotation by $\ta_N:=2\pi/N$ about 0, that is, $R(z):=\zeta z$ for all $z\in\C$.
Recall that $V_0$ is a sector with vertex $1=\zeta^0$ formed by lines
 containing the two incident to $1$ edges of $\Pi$, see Fig. \ref{fig:dual}.
Then $V_k=R^k(V_0)$; here the index $k$ can be viewed as a residue modulo $N$.
Write $S_k$ for the central reflection (=half-turn) about the vertex $\zeta^k$; we have
$$
S_k(z)=2\zeta^k-z=R^k\circ S_0\circ R^{-k}(z).
$$

Consider the strip formed by the lines $1\zeta$ and $(-1)(-\zeta)$; call it \emph{strip 0}; see Fig. \ref{fig:8gon}.
The line $(-\zeta)(-\zeta^{2})$ divides strip 0 into two equal unbounded parts.
One part, call it the \emph{senior part}, contains the polygon $\Pi$.
Define the \emph{vassal part} as the other part of strip 0; it is centrally symmetric to
 the senior part with respect to the midpoint of the common edge.
A regular polygon $\Pi^\dagger$ in the vassal part is symmetric to $\Pi$;
 call it the \emph{vassal polygon} of $\Pi$.
Since the two polygons $\Pi$ and $\Pi^\dagger$ are symmetric to each other
 with respect to the midpoint of the common boundary between the senior
 and the vassal parts of strip 0, these polygons have the same sidelength $\la_0=|\zeta-1|$.

Recall that $N$ is even, hence it has the form $N=2m+2$.\footnote{While the notation $N$
 is \emph{global}, i.e., used throughout the paper, the notation $m$ is \emph{local} --- it is effective till the
 end of Section \ref{s:bnded}, after which $m$ becomes a ``free variable'' again.}
Let $c_m$ be the intersection point of the line $(-\zeta)(-\zeta^2)$ with the line $1\zeta$.
Polygons $\Pi$ and $\Pi^\dagger$ are symmetric with respect to the midpoint of the segment $[-\zeta;c_m]$.
It follows that the line $(-\zeta)(-\zeta^2)$ contains an edge $[c_{m-1};c_m]$ of $\Pi^\dagger$.
Write $c_0$, $\dots$, $c_m$ for the vertices of $\Pi^\dagger$ symmetric to $1$, $\zeta^{-1}$, $\dots$, $\zeta^{-m}=-\zeta$,
 so that $c_k$ is symmetric to $\zeta^{-k}$.
Both $\Pi$ and $\Pi^\dagger$ are symmetric with respect to the reflection axis of strip 0.
Hence the line $1\zeta^{-1}$ contains the edge $[c_0;c_1]$ of $\Pi^\dagger$.

\begin{lem}
  The polygons $\Pi^\dagger$ and $R^{-1}\Pi^\dagger$ have $c_0$ as a common vertex.
\end{lem}

\begin{proof}
See Fig. \ref{fig:8gon}, left.
Define \emph{strip $-1$} as the image of strip 0 under $R^{-1}$.
It is bounded by the lines $\zeta^{-1}1$ and $(-1)(-\zeta)$.
Note that $R^{-1}c_m=c_0$ since $c_0$ is on the boundary line $\zeta^{-1}1$ of strip $-1$
 corresponding to the boundary line $1\zeta$ of strip 0 under $R^{-1}$ and,
 moreover, the distance from $c_0$ to $\zeta^{-1}$ equals the distance from $c_m$ to 1.
It follows that the vertex $c_0$ of $\Pi^\dagger$ is also a vertex of the polygon $R^{-1}\Pi^\dagger$, as claimed.
\end{proof}

\begin{figure}
  \centering
  \includegraphics[height=6cm]{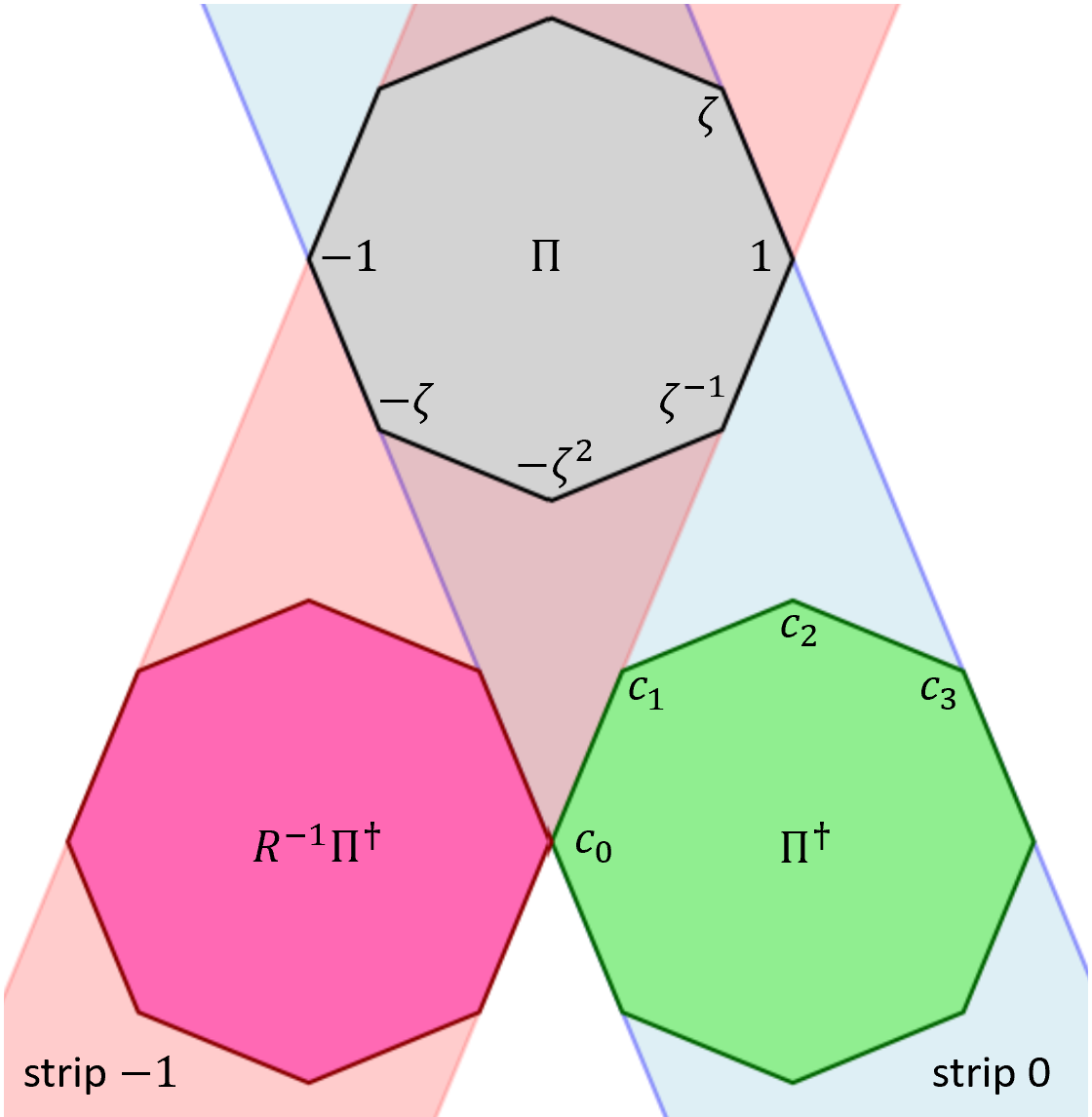}
  \includegraphics[height=6cm]{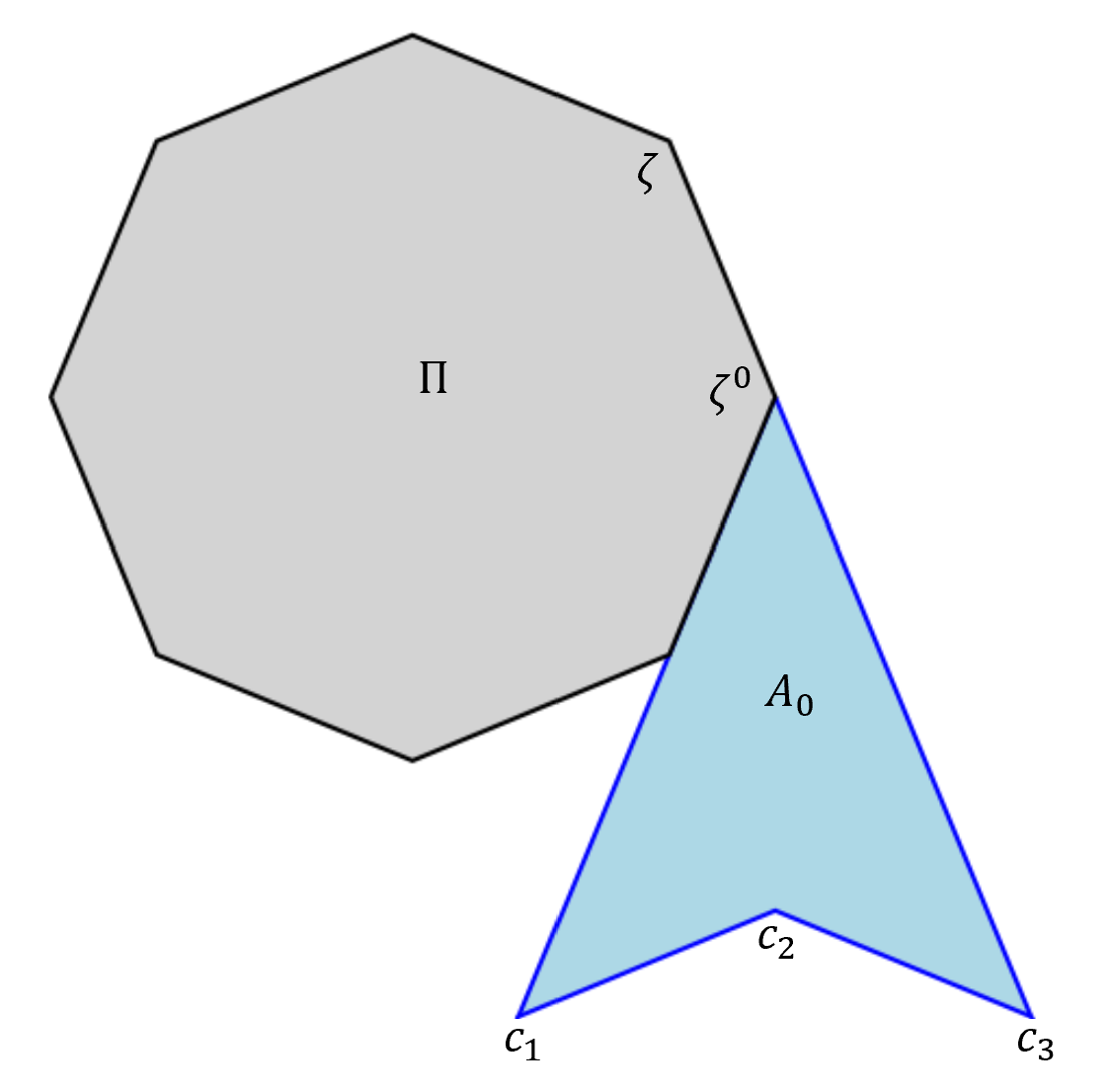}
  \caption{\small
  Left: the vassal polygon $\Pi^\dagger$, strip 0 and strip $-1$.
  Right: the polygon $\Pi$ and the subpolygon $A_0$ for $N=8$.}\label{fig:8gon}
\end{figure}

It is easy to characterize the vassal polygon in terms of dynamics:
 the following lemma is straightforward and is left to the reader.

\begin{lem}
\label{l:vassal-dyn}
The interior of
 $\Pi^\dagger$ of $\Pi$ coincides with the set of all points $z\in V_0$
 such that $f_\Pi^i(z)\in V_{mi\pmod N}$ for all $i\in\Z$.
\end{lem}

\subsection{The first invariant neighborhood $\Upsilon$ of $\Pi$}
\label{ss:fin}
The $R$-orbit of $\Pi^\dagger$ consists of $N$ regular polygons touching each other in a cyclic order,
 forming a \emph{necklace} (terminology is inspired by \cite{GS92} even though it is
 not the same necklace as there).

\begin{dfn}[The first invariant neighborhood]
\label{d:fin}
Let $\Upsilon$ be the bounded complementary component of this necklace.
Call $\Upsilon$ \emph{the first invariant neighborhood} of $\Pi$.
It is clear that $\Pi\subset \Upsilon$, and the term ``invariant'' refers to the fact
 that $f_\Pi$ is a scissors automorphism of $\Upsilon\sm\ol\Pi$
 (alternatively, one can extend $f_\Pi$ to all of $\Upsilon$ by setting it to
 act as the identity map on $\Pi$).
\end{dfn}

Set $A_0:=(V_0\sm\ol \Pi^\dagger)\cap \Upsilon$; this is a non-convex polygon with
 vertices $1$, $c_1$, $\dots$, $c_{m}$; see Fig. \ref{fig:8gon} where $m=3$.
Here, the numbering of the vertices follows their counterclockwise order in the boundary of $A_0$,
 see Fig. \ref{fig:8gon}, right.

Note that any diameter line of $\Pi$ is a symmetry axis of $\Upsilon$,
 and that $f_\Pi$ takes $A_0$ to the symmetric polygon with respect to the diameter through $1$ and $-1$.
In particular, $f_\Pi(A_0)\subset \Upsilon$, and similarly for all $A_k:=R^kA_0$ for $k=0$, $\dots$ $N-1$.
The restriction $f_\Upsilon$ of $f_\Pi$ to $\Upsilon\sm\ol\Pi$ is therefore a piecewise rotation of $\Upsilon\sm\ol\Pi$ whose
 domain consists of the polygons $A_k$, cf.  \cite[Lemma 20]{Ruk22}.
See Fig. \ref{fig:gX}, left, where these ``cursor'' shaped polygons $A_k$ are shown for $N=8$.

\begin{figure}
  \centering
  \includegraphics[width=.8\textwidth]{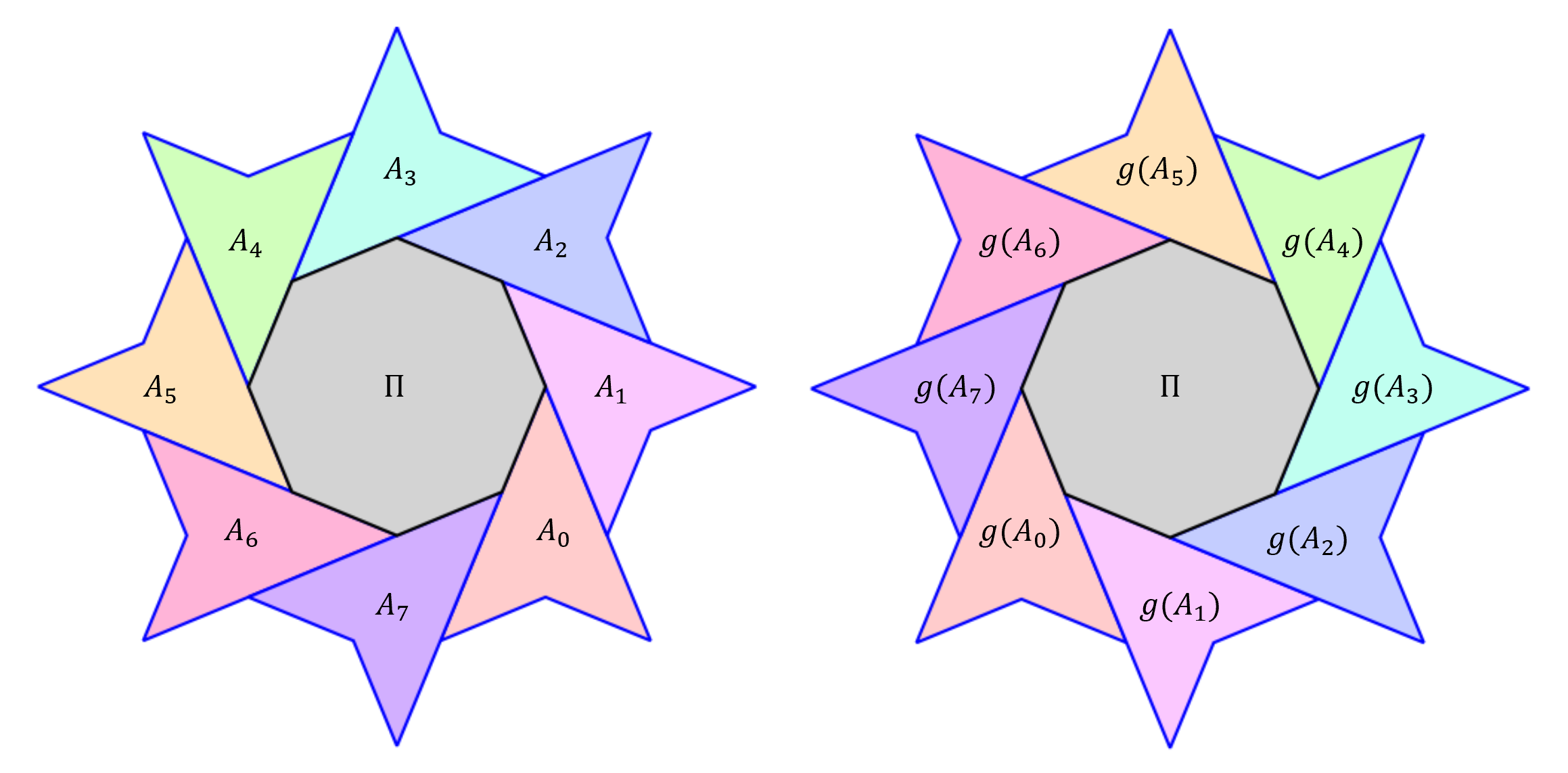}
  \caption{\small The action of the map $g=g_\Upsilon$ on $\Upsilon\sm\Pi$.
  Left: components of $\dom(g)$.
  Right: their $g$-images.
  }\label{fig:gX}
\end{figure}

Let $M$ denote the half-turn about the origin, that is, the map $z\mapsto -z$;
 one can also write $M=R^{N/2}$.
Since $f_\Upsilon$ acts as a half-turn on every component of $\dom(f_\Upsilon)$,
 the map $g_\Upsilon:=M\circ f_\Upsilon$ is clearly a polygon exchange transformation.
Fig. \ref{fig:gX} illustrates the action of $g_\Upsilon$.
Observe that $M$ conjugates $f_\Upsilon$ with itself, therefore, it takes orbits to orbits,
 hence $g_\Upsilon$ has aperiodic points if and only if $f_\Upsilon$ does.

\begin{thm}
\label{t:fX-nonper}
Let $L$ be the line through the vertices $\zeta^0=1$ and $\zeta^1$ of $\Pi$.
Then $\Inv_L(g_\Upsilon)\ne 0$.
It follows that the scissors automorphisms $g_\Upsilon$ and $f_\Upsilon$ of $\Upsilon$ are not periodic.
\end{thm}

Theorem \ref{t:fX-nonper} is proved in the next section.

\subsection{Computation of $\Inv_L(g_\Upsilon)$}
All tensor products considered below are over $\Q$.
For our subsequent computation of the dynamic Hadwiger invariant for $g_\Upsilon$,
 we will need to know the lengths of all edges of $A_{0}$.
Edges $[c_1;c_2]$, $\dots$, $[c_{m-1};c_m]$ all have the same length as an edge of $\Pi$.
Namely, their length is $\la_0=2\sin\frac{\ta}{2}$.
Side edges $[1,c_1]$ and $[1;c_m]$ of $A_0$ have the same length $\la'$, which we will now compute.
Consider the triangle with vertices $1$, $-1$, $c_0$; this is an isosceles triangle
 with base length 2 and angles $\ta$, $\frac{\pi-\ta}{2}$, $\frac{\pi-\ta}{2}$.
It follows that the side edge of this triangle has length
$$
\la_0+\la'=\frac{1}{\sin\frac\ta 2}=\frac{2}{\la_0},
$$
which gives an expression for $\la'$ in terms of $\la_0$.
In the sequel, we will need an expression for $\la'$ in terms of $\la_0$ and $\zeta$.
One can use the relation $\zeta+\zeta^{-1}=2-\la_0^2$, in which the left-hand side
 is equal to $2\cos\ta$, and the right-hand side follows from the well-known expression for $\cos\ta$
 in terms of $\sin\frac{\ta}{2}$.
As a consequence, we obtain that
$$
\la'=\la_0\frac{\zeta+\zeta^{-1}}{2-\zeta-\zeta^{-1}}.
$$

Recall from the statement of Theorem \ref{t:fX-nonper} that
 the line $L$ connects $1$ and $\zeta$,
 and define the positive side of $L$ as the one \emph{not} containing $\Pi$.
Each edge $e$ of $A_0$ gives rise to two $L$-adapted pairs (see Section \ref{ss:val}) $(A_k,E)$ and $(M(A_k),M(E))$,
 where $E=R^ke$ is the corresponding to $e$ edge of $A_k$ parallel to $L$, and $M$ is the map $z\to -z$.
Call the pairs $(A_k,E)$ and $(MA_k,ME)$ the $L$-adapted pairs \emph{corresponding} to $e$.
Among these two pairs, exactly one is positive, and the other one is negative.
For example, $(A_0,[c_m;1])$ is a negative pair, hence, the corresponding to the edge $[c_m;1]$
 positive pair is $(A_{m+1},M[c_m,1])$.
Recall that $N=2m+2$.
Positive pairs are described in the next lemma.

\begin{lem}
  \label{l:pospair}
All positive $L$-adapted pairs
 are
$$
(A_1,R[1,c_1]),\quad (A_{m+1},M[c_m,1]),
$$
$$
(A_{-k},R^{-k}[c_{m-k},c_{m-k+1}]),\quad k=1,\dots,m-1.
$$
\end{lem}

\begin{proof}
In order to obtain a segment parallel to $L$, rotate $[c_{m-1};c_m]$ by $R^{-1}$.
Clearly, $A_{-1}$ appears on the positive side of $R^{-1}[c_{m-1};c_m]$.
Next, each time we pass from $k$ to $k+1$, the next segment $[c_{m-k-1};c_{m-k}]$
 is rotated by angle $-\ta$ compared to $[c_{m-k};c_{m-k+1}]$,
 and the side on which $A_0$ lies rotates accordingly.
This implies the desired result.
\end{proof}

We also need to know the translation vectors of $g_\Upsilon$.

\begin{lem}
  \label{l:tvec}
  Each polygon $A_k$ is translated by $\vec v_k=-2\zeta^k$ under $g_\Upsilon$.
\end{lem}

\begin{proof}
Under the map $f_\Upsilon$, or, more precisely, the continuous extension of it to $\ol A_k$,
 the vertex $\zeta^k$ of $V_k$ is mapped to itself.
Next, $M$ takes $\zeta^k$ to $M\zeta^k=-\zeta^k$.
Therefore, the continuous extension of $g_\Upsilon|_{A_k}$ onto $\ol A_k$ takes $\zeta^k$ to $-\zeta^k$.
In other words, $g_\Upsilon$ translates the polygon $A_k$ by vector $\vec v_k=(-\zeta^{k})-\zeta^k=-2\zeta^k$.
\end{proof}

Now that all required edge lengths and translation vectors are found,
 we can complete the computation of the dynamic Hadwiger invariant.

\begin{thm}
\label{t:Inv-gX}
Suppose $N>2$ is even.
The dynamic Hadwiger invariant of $g_\Upsilon$ equals
$$
\Inv_L(g_\Upsilon)=4\la_0\otimes\frac{\zeta(\zeta+\zeta^{-1})}{1-\zeta}-4\la_0\frac{\zeta(\zeta+\zeta^{-1})}{(1-\zeta)^2}\otimes (1-\zeta).
$$
\end{thm}

In the expressions for $\Inv_L(g_\Upsilon)$ given above,
 we used the following convention: if $x$ and $z$ are products
 of complex numbers, then these products are evaluated first,
 before evaluating $x\otimes z$.
The same convention is used throughout the paper.

\begin{proof}
Consider an edge of $A_0$ and the two corresponding $L$-adapted pairs.
Observe that the corresponding polygons $A_k$ and $MA_k$ are translated by
 opposite vectors under $g_\Upsilon$, by Lemma \ref{l:tvec}.
For this reason, $\Inv_L(g_\Upsilon)$ is twice the sum of the contributions of all positive pairs.
Edges of $A_0$ not on the boundary of $V_0$ contribute the sum
 $2\la_0\otimes \vec v_{\mathrm{bot}}$, where $\vec v_\mathrm{bot}$ is the sum of $\vec v_{-k}$
 over all $k=1$, $\dots$, $m-1$, and the subscript `bot' stands for `bottom'.

Computing $\vec v_{\mathrm{bot}}$ amounts to summing a geometric series, namely,
$$
\vec v_\mathrm{bot}=-2\sum_{k=1}^{m-1} \zeta^{-k}=-2\zeta^{-m}\frac{\zeta-\zeta^m}{1-\zeta}=2\frac{\zeta(\zeta+\zeta^{-1})}{1-\zeta}.
$$
It remains to account for the contributions of the two edges of $A_0$ that lie on the boundary of $V_0$.
Write this last contribution as $2\la'\otimes \vec v_{\mathrm{top}}$; where $\vec v_\mathrm{top}$ is equal to
 $\vec v_0+\vec v_1=2-2\zeta$.
Using the expression for $\la'$ through $\la_0$, we obtain the desired expression for $\Inv_L(g_\Upsilon)$.
\end{proof}

It remains to verify that the invariant is non-zero.

\begin{proof}[Proof of Theorem \ref{t:fX-nonper}]
According to the proof of Theorem \ref{t:Inv-gX}, the invariant $\Inv_L(g_\Upsilon)$
 can be written as $2\la_0\otimes \vec v_\mathrm{bot}+2\la'\otimes \vec v_\mathrm{top}$.
Here $\vec v_\mathrm{bot}$ and $\vec v_\mathrm{top}$ are complex numbers such that
$$
-\frac{\vec v_\mathrm{bot}}{\vec v_\mathrm{top}}=\frac{\la'}{\la_0}=\frac{\cos\ta}{2-\cos\ta}.
$$
This is a real but not rational number, under our assumptions on $N$.
Hence, $\vec v_\mathrm{bot}$ and $\vec v_\mathrm{top}$ are linearly independent over $\Q$.
Let $\Vc$ be the rational vector space spanned by these two vectors.
Note that $\Inv_L(g_\Upsilon)$ lies in $\R\otimes\Vc$, and the latter group has a natural
 structure of a real vector space with basis $1\otimes \vec v_\mathrm{bot}$ and $1\otimes \vec v_\mathrm{top}$.
Since $\Inv_L(g_\Upsilon)$ is a non-trivial linear combination of the basis vectors,
 this element is nonzero in $\R\otimes\Vc\subset \R\otimes\C$.

By Theorem \ref{t:additiv}, a periodic polygon exchange transformation
 of a bounded polygon must have zero invariant $\Inv_L$.
It now follows $\Inv_L(g_\Upsilon)\ne 0$ that $g_\Upsilon$ is not periodic.
Hence $f_\Upsilon=M\circ g_\Upsilon$ is not periodic either, since $M$ has period 2 and
 commutes with $f_\Upsilon$ and $g_\Upsilon$.
\end{proof}

\subsection{An informal explanation of Theorem \ref{t:Rukh}}
\label{ss:Rukh}
Aiming at providing a more complete understanding of the principal ideas behind the Main Theorem,
 we informally overview the proof of Theorem \ref{t:Rukh} here.
More precisely, we explain only the part of it that is needed for the Main Theorem:
 if the outer billiard on $\Pi_{2N}$, where $N$ is odd, has an aperiodic point, 
 then so does the outer billiard on $\Pi_N$.
The reader is referred to \cite{Ruk22} for formal details.

\begin{figure}
  \centering
  \includegraphics[width=\textwidth]{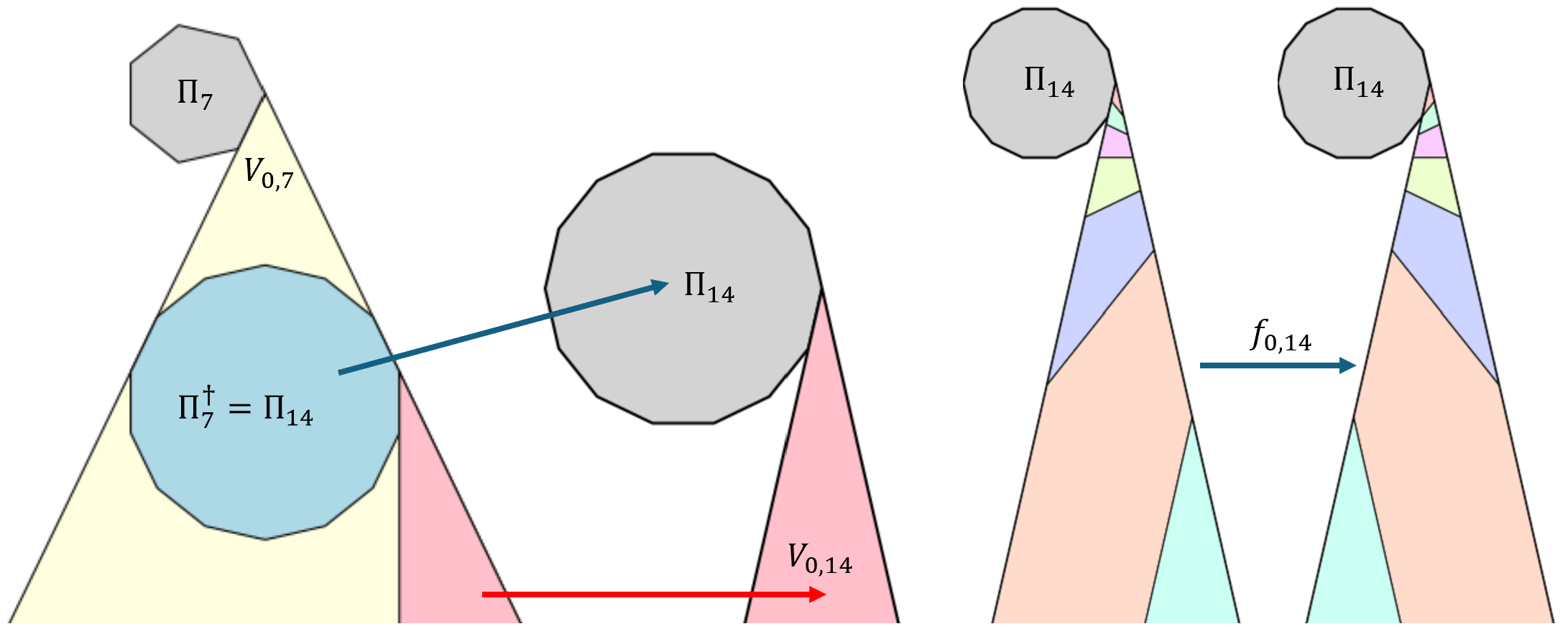}
  \caption{\small Left: a relationship between the outer billiard on $\Pi_N$ and that on $\Pi_{2N}$,
   for $N$ odd. Here, the case $N=7$ is shown.
   Right: the induced map $f_{0,14}$; note that it acts as a piecewise rotation
   rather than as an axial reflection.}\label{fig:Rukh}
\end{figure}

\begin{figure}
  \centering
  \includegraphics[width=\textwidth]{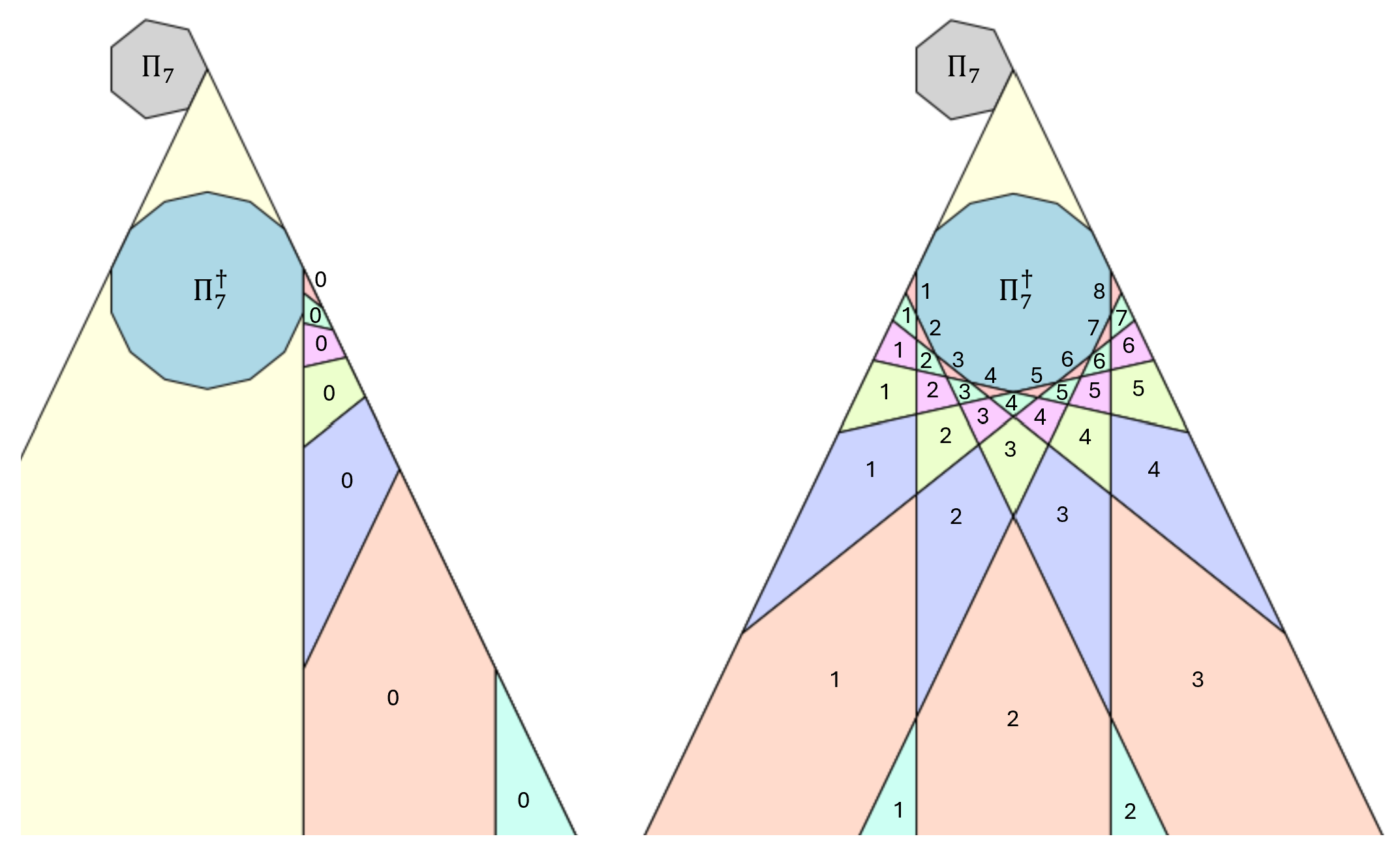}
  \caption{\small Components of $\dom(f_{0,14})$ are mapped forward under $f_{0,7}$ as shown.
  Every such component is marked with 0 on the left, and its iterated forward $f_{0,7}$-images
  are marked with positive integers (representing the iterates) on the right.
  For example, a small triangle marked with 0 is mapped, under $f^3_{0,7}$, to a small triangle of
  the same shape marked with 3.
  The return time to $V_{0,14}$ is the greatest mark associated with each shape.}\label{fig:Rukh1}
\end{figure}

Suppose that $N$ is odd, and consider a regular $N$-gon $\Pi_{N}$.
The vassal polygon of $\Pi_N$ can also be defined for odd $N$ but, in this case, it is
 a regular $2N$-gon rather than an $N$-gon.
Recall that $V_k=V_{k,N}$, where $k\in\Z/N\Z$, is the sector with the vertex at $\zeta^k_N=\exp(2\pi\mathbf{i}k/N)$, on which
 the outer billiard map $f_N:=f_{\Pi_N}$ acts as the reflection in $\zeta^k_N$.
Set now $m=[N/2]$ so that $N=2m+1$ and define the \emph{vassal polygon}
 $\Pi^\dagger_{N}$ of $\Pi_N$ as the regular $2N$-gon inscribed simultaneously into $V_{0,N}$ and $V_{-m,N}$,
 see Fig. \ref{fig:Rukh}.
It is also straightforward to characterize $\Pi^\dagger_{N}$ as the set of points $z$
 such that $f_N^i(z)\in V_{im,N}$, for all $i\in\Z$, similarly to Lemma \ref{l:vassal-dyn} above.

Re-denote the $2N$-gon $\Pi^\dagger_{N}$ as $\Pi_{2N}$
 and consider the outer billiard $f_{2N}$ on $\Pi_{2N}$.
Define the \emph{induced map} $f_{0,2N}$ of $f_{2N}$ as follows.
Given any component $P$ of $\dom(f^2_{2N})$, there is a unique $k$ with $f_{2N}(P)\subset V_{k,2N}$,
 and we let $f_{0,2N}$ act on $P$ as $R^{-k}_{2N}\circ f_{2N}$, where $R_{2N}$ is the rotation by angle $\theta_{2N}=\pi/N$
 about the center of $\Pi_{2N}$.
The map $f_{0,2N}$ thus defined is a scissors automorphism of $V_{0,2N}$; see Fig. \ref{fig:Rukh} (right).
Similarly, one defines the induced map $f_{0,N}$ of $f_N$; the latter is a scissors automorphism of $V_{0,N}$.
A crucial observation is: \emph{the first return map of $V_{0,2N}$ under $f_{0,N}$ coincides with $f_{0,2N}$.}
See Fig. \ref{fig:Rukh1} for an illustration.
It follows that all aperiodic points of $f_{2N}$ in $V_{0,2N}$, which are also
 aperiodic points of $f_{0,2N}$, are simultaneously aperiodic points of $f_{0,N}$, and hence of $f_N$.

\section{Julia sets in flag spaces}
\label{s:flag}
\def\Fl{\mathrm{Fl}}
It follows from Theorem \ref{t:fX-nonper} that $f_\Upsilon$ has either aperiodic points or
 periodic points of arbitrarily large period.
We will prove that, in the latter case, one can find a converging sequence of periodic points
 whose periods tend to infinity such that the limit of the sequence is an aperiodic point.

Consider a polygon exchange transformation $f$ of a bounded open polygon $\Delta$.
Say that $f$ is \emph{unbranched} if $\Delta\sm\dom(f)$ is a union of finitely
 many straight line intervals with pairwise disjoint closures.
Note that $g_\Upsilon$ acts as an unbranched polygon exchange transformation of $\Upsilon\sm\ol\Pi$,
 as is shown in Fig. \ref{fig:unbd}, left.
The Main Theorem will follow from Theorem \ref{t:aper} stated below.

\begin{figure}
  \centering
  \includegraphics[height=5cm]{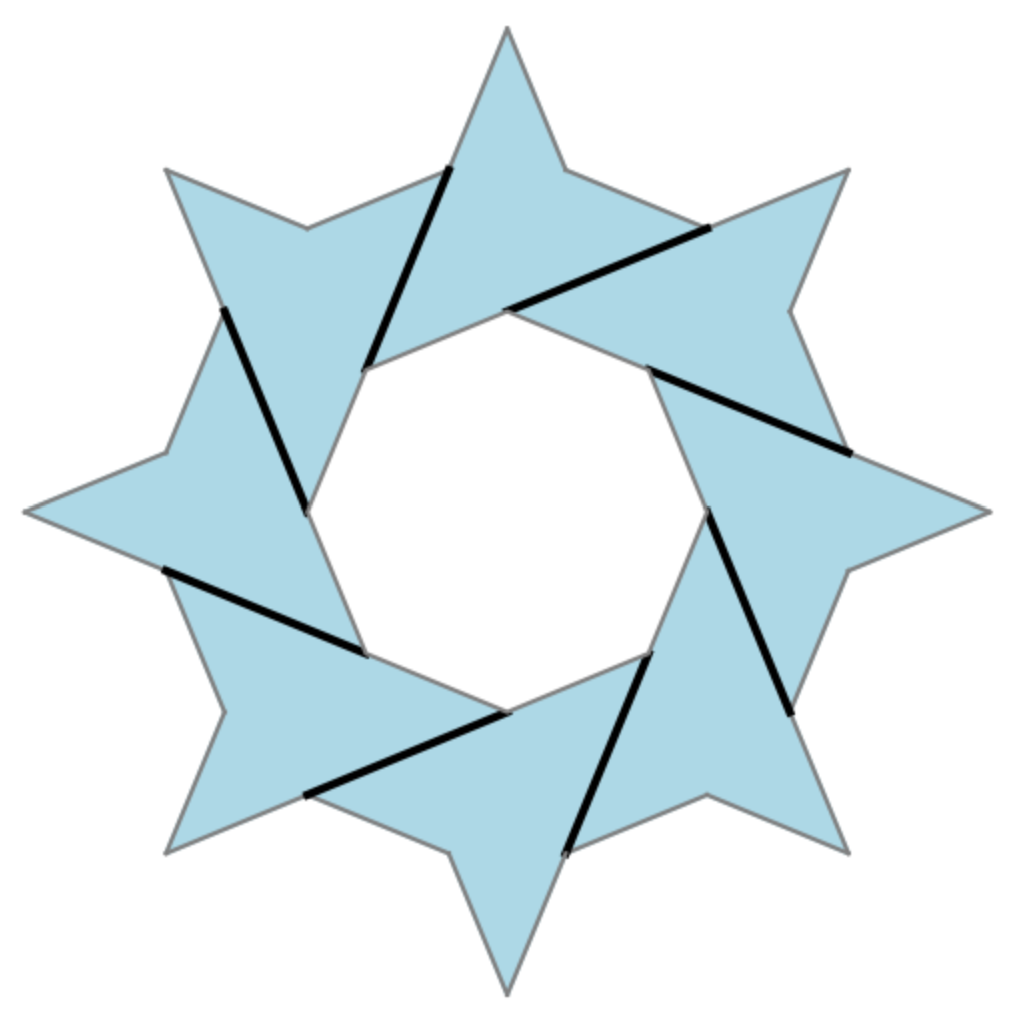}\hspace{.5cm}
  \includegraphics[height=5cm]{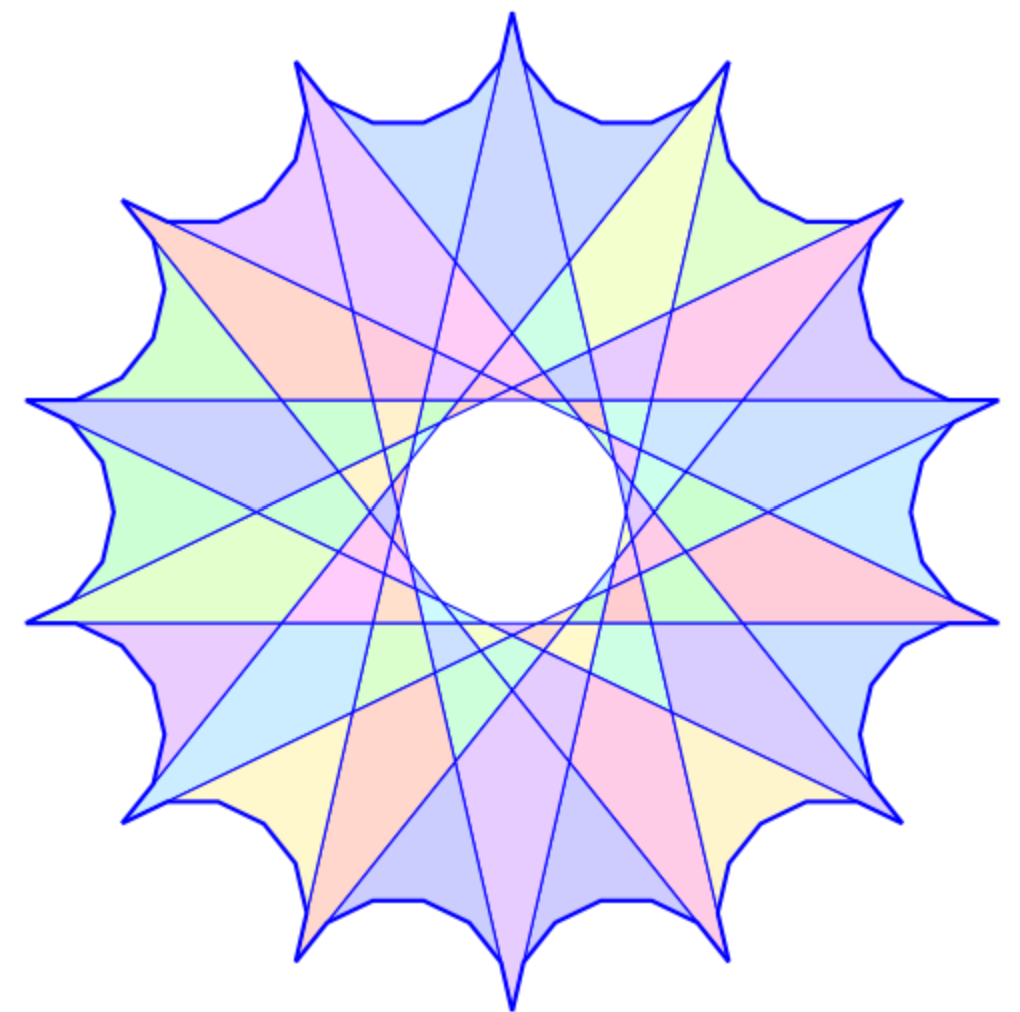}
  \caption{\small Left: for $\Delta=\Upsilon\sm\ol\Pi$ in the case $N=8$, the figure shows (in black) the set
  $\Delta\sm\dom(g_\Upsilon)$,
  where $g_\Upsilon$ is as in Section \ref{ss:fin}; this set is a union of disjoint straight
  line intervals. Thus, $g_\Upsilon$ is unbranched. Right: all level 1 pieces of $\Upsilon\sm\ol\Pi$ in the case $N=14$,
  i.e., the complementary components of $\Gamma_1=\Gamma_1^+\cup\Gamma_1^-$.}\label{fig:unbd}
\end{figure}

\begin{thm}
  \label{t:aper}
Let $f$ be an unbranched polygon exchange transformation of a bounded polygon
 such that, for at least one co-oriented line $L$,
the dynamic Hadwiger invariant $\Inv_L(f)$ is nonzero.
Then $f$ has aperiodic points.
\end{thm}

Here, the assumption that $f$ is unbranched can probably be removed;
 however, it simplifies the arguments, and it does hold for polygon exchange transformations
 arising from outer billiards on regular polygons.
Whether the assumption about a dynamic Hadwiger invariant being nonzero is essential, we do not know.
Theorems \ref{t:aper} and \ref{t:fX-nonper} imply the Main Theorem,
 since the map $g_\Upsilon$ from Section \ref{s:bnded} is clearly unbranched and has non-vanishing dynamic Hadwiger invariants.

\subsection{Flag topology}
\label{ss:fl-top}
A \emph{flag} in $\R^2$ is by definition a triple $(a,R,H)$, where $a\in\R^2$ is a point,
 $R$ is a straight ray originating at $a$, and $H$ is a halfplane containing $R$ in its boundary.
Say that a flag $(a,R,H)$ \emph{belongs} to a compact polygon $\ol P$ if $a\in\ol P$,
 the intersection $R\cap\ol P$ has dimension one, and the intersection $H\cap\ol P$ has dimension two.
Write $\Fl(\ol P)$ for the set of all flags belonging to $\ol P$;
 these flags are also said to \emph{lie in} $\ol P$.

In this section, we introduce a certain topology on spaces of flags.
Fix an open bounded polygon $\Delta$ in $\R^2$, not necessarily convex or connected,
 and consider the set $\Fl(\ol\Delta)$.
Subsets $\Fl(\ol P)\subset\Fl(\ol\Delta)$, where $P$ ranges through open convex polygons in $\Delta$,
 form a basis of some topology on $\Fl(\ol\Delta)$, the so-called \emph{flag topology}.
A principal motivation for introducing flag topology is that scissors
 automorphisms of $\Delta$ act as homeomorphisms of $\Fl(\ol\Delta)$,
 with respect to this topology.

\begin{prop}
For every Euclidean scissors congruence $f$ between open bounded polygons $\Delta$ and $\Delta'$, the
 associated map $f_*:\Fl(\Delta)\to\Fl(\Delta')$ is a homeomorphism.
\end{prop}

\begin{proof}
Assume that the components of $\dom(f)$ are polygons $P_1$, $\dots$, $P_n$.
For every $i=1$, $\dots$, $n$ and every compact subpolygon $\ol P\subset \ol P_i$,
 the map $f_*$ takes $\Fl(\ol P)$ bijectively onto $\Fl(\ol f(P))$
 (here, $P$ is the interior of $\ol P$).
Therefore, the restriction of $f_*$ to $\Fl(\ol P_i)$ is a homeomorphism onto its image.
Since $\Fl(\Delta)$ is the disjoint union of the clopen subsets $\Fl(\ol P_i)$, the 
 the proposition follows.
\end{proof}

In one-dimensional setting, there is the following analog of flag topology.
Double each point $a$ of the unit circle $\uc$ so that to obtain two points $a_\pm$ called \emph{flags of $\uc$};
 write $\Fl(\uc)$ for the set of all such flags.
Given any compact circle arc $A\subset \uc$, define $\Fl(A)$ as the set
 of all flags $a_\pm$, where $a\in A$, and the sign is necessarily plus (resp., minus)
 if $a$ is the beginning of $A$ (resp., the end of $A$) with respect to the standard orientation of $\uc$.
Sets $\Fl(A)$ form a basis of some topology on $\Fl(\uc)$, as $A$ ranges through all compact arcs of $\uc$.
Note that the natural projection from $\Fl(\uc)$ onto $\uc$ is \emph{perfect} (continuous, closed, and with compact fibers).
It follows from the compactness of $\uc$ that $\Fl(\uc)$ is also compact; in fact,
 $\Fl(\uc)$ is a \emph{Stone space}, that is, a compact 
  totally disconnected space.
Topological spaces like $\Fl(\uc)$ have been used earlier to study groups of
 interval exchange transformations (a similar definition first appeared in
 \cite[\S 5]{Kea75}, see also \cite{JBMdlS18,Lac21} for more recent applications).

\begin{thm}
\label{t:fl-comp}
The flag topology on $\ol\Delta$ is compact, Hausdorff, and totally disconnected;
 thus, $\Fl(\ol\Delta)$ equipped with this topology is a Stone space.
\end{thm}

\begin{proof}
We prove that $\Fl(\ol\Delta)$ is compact and leave the other claims to the reader.
Let $\Uc$ be an open cover of $\Fl(\Delta)$; we want to find a finite subcover of $\Uc$.
Assume without loss of generality that all $U\in\Uc$ have the form $U=\Fl(\ol Q)$ for some open polygon $Q\subset \Delta$.
Choose a dissection of $\Delta$ into polygons of diameter $<1$.
There is at least one open polygon $P_1$ of this dissection with the property that
 $\Uc$ does not admit any finite subcover over $\Fl(\ol P_1)$.
Next, choose a dissection of $P_1$ into polygons of diameter $<1/2$,
 and let $P_2$ be an open polygon of the latter dissection such that $\Uc$ does not have
 any finite subcover over $\Fl(\ol P_2)$.
One can now continue this process.
As a result, one obtains a nested ($P_{n+1}\subset P_n$) sequence of open polygons $P_n$
 whose diameter tends to zero.
Hence the intersection of all $\ol P_n$ is a single point $a\in S$.

The unit tangent space $\UT_a$ of $\R^2$ at a point $a$, that is, the set of unit tangent vectors at $a$,
 is homeomorphic to the unit circle $\uc$; we will identify the two spaces.
One can also identify $\UT_a$ with the set of all \emph{directions}, i.e., germs
 of rays originating at $a$.
The subset of directions contained in $\ol P_n$ for all $n$ is a certain compact arc $A\subset\UT_a=\uc$.
Note that $A$ may be degenerate (i.e., a single point) or the entire circle.
Let $E\subset A$ be the set of endpoints of $A$, which may be the empty set (if $A=\uc$),
 a singleton (if $A$ is degenerate), or a two point set (if $A$ is a nondegenerate arc).
Consider the set $\Fl_a(\ol\Delta)$ of all flags of $S$ of the form $(a,R,H)$, with fixed $a$ and variable $R$, $H$,
 and define the natural homeomophism between $\Fl_a(\ol\Delta)$ and $\Fl(\uc)$ that takes $(a,R,H)$ to
 the flag $R_\pm\in\Fl(\uc)$,
 where $R\in\UT_a$ is identified with a point of $\uc$, and the side of $R$ in $\uc$ (equivalently,
 the subscript $\pm$ of $R_\pm$) is defined by $H$.
Observe that the intersection of all $\Fl(\ol P_n)$ is a closed subset of $\Fl(\uc)$, hence it is compact.
More precisely, $Z:=\bigcap_n \Fl(\ol P_n)$ is the union of $\Fl(A)$ and at most two flags of $\uc$
 based at points of $E$ (flags that ``look'' outside of $A$ may still be contained in all $\Fl(\ol P_n)$).

By compactness of $Z$, there is a finite subcover $\{U_1,\dots,U_m\}$ of $\Uc$ over $Z$.
Also, each $U_i$ has the form $\Fl(\ol Q_i)$, for some open polygon $Q_i\subset S$.
The germ of the polygon $\ol Q_1\cup\dots\cup\ol Q_m$ at $a$ covers the germ of $\ol P_n$,
 for sufficiently large $n$.
It follows that $\Fl(\ol P_n)$ is covered by $U_1\cup\dots\cup U_m$, a contradiction
 with our assumption that $\Uc$ has no finite subcover over $\ol P_n$.
\end{proof}

We will often use the map from $\Fl(\ol\Delta)$ onto $\ol\Delta$ that takes $(a,R,H)$ to $a$;
 call this map \emph{the natural projection}.

\subsection{The flag Julia set}
\label{ss:Jset-Fl}
Let $\Delta$ be a bounded open polygon in $\R^2$, and consider a Euclidean scissors automorphism $f$ of $\Delta$.
Define the \emph{flag Julia set} $\Jc_f$ of $f$ as the subset of $\Fl(\ol\Delta)$ formed by all flags that
 are not periodic under the action of $f_*$.
The complement $\Fc_f:=\Fl(\ol\Delta)\sm\Jc_f$ is called the \emph{flag Fatou set}.
Evidently, if $q\in\Fc_f$, then some neighborhood of $q$ consists of periodic flags of the same period as $q$.
Hence, $\Fc_f$ is open, and $\Jc_f$ is compact.
Our terminology is motivated by an analogy with the dynamics of rational functions of a single complex variable.


\begin{prop}
\label{p:Jc-empty}
Consider $f\in\PET(\Delta)$, for a bounded open polygon $\Delta$.
The Julia set $\Jc_f$ is empty if and only if $f$ is periodic.
\end{prop}

\begin{proof}
Suppose first that $f$ is periodic.
Then all flags are periodic, therefore, $\Jc_f=\0$ by definition of $\Jc_f$.
Assume now that $\Jc_f=\0$, that is, all flags $q\in\Fl(\ol\Delta)$ are periodic.
Then the period of $q$ is a continuous function of $q$,
 and a continuous function on a compactum attains its maximal value, say, $n$.
Hence $f^n_*$ fixes every flag of $\ol\Delta$, which means that $f^n$ is the identity element in the group $\PET(\ol\Delta)$.
\end{proof}

As in the proof of Theorem \ref{t:fl-comp},
 let $\Fl_a(\ol\Delta)$ stand for the set of all flags $q\in\Fl(\ol\Delta)$
 that map to $a$ under the natural projection,
 i.e., that have form $q=(a,R,H)$, for some unspecified $R$ and $H$.

\begin{dfn}[The level function]
Define $\Gamma_f:=\ol\Delta\sm (\dom(f)\cap\dom(f^{-1}))$ as the set of all points of $\ol\Delta$
 where either $f$ or $f^{-1}$ is undefined.
Given a flag $q\in\Fl_a(\ol\Delta)$, let $\lev(q)$, the \emph{level of $q$}, be the
 maximal number $n\in \Z_{\ge 0}\cup\{\infty\}$ with the property that $f^k(a)$ are defined, for all $|k|<q$.
In particular, $q$ has level one if and only if $a\in\Gamma_f$,
 and $q$ has level $<\infty$ if and only if $a$ is a boundary point of $f$.
The function
$$
 \lev:\Fl(\ol\Delta)\to\Z_{\ge 0}\cup\{\infty\}
$$
is called the \emph{level function}.
\end{dfn}

We see that the level function is in fact defined on $\ol\Delta$ and
 then transferred to $\Fl(\ol\Delta)$ by means of the natural projection.
Note that the level function is lower semicontinuous
(a function $\xi:X\to\R\cup\{+\infty\}$ on a topological space $X$ is said to be \emph{lower semi-continuous}
 if the set $f^{-1}\{x\mid x>c\}$ is open, equivalently, the set $f^{-1}\{x\mid x\le c\}$ is closed, for every $c\in\R$).
Indeed, $\{q\in\Fl(\ol\Delta)\mid \lev(q)\le n\}$ is the complement in $\Fl(\ol\Delta)$ of the open set
 $\bigcup_{a}\Fl_a(\ol\Delta)$, where $a$ ranges through $\dom(f^n)\cap\dom(f^{-n})$.

If $\lev(q)=\infty$ for some $q\in\Jc_f$, this means that $q\in\Fl_a(\ol\Delta)$,
 for an aperiodic point $a\in\Delta$, therefore, aperiodic points exist.
Arguing by contradiction, we will assume that $f$ is an unbranched polygon exchange transformation of $\Delta$
 such that $f$ is non-periodic but does not have aperiodic points.
In this case, by Proposition \ref{p:Jc-empty}, the Julia set $\Jc_f$ is nonempty while
 the level function takes only finite values on $\Jc_f$.

\begin{rem}[Local unboundedness of $\lev$]
It may be reasonable to conjecture that, in general, the restriction of the level function to $\Jc_f$ is \emph{locally unbounded},
 that is, $\lev$ is unbounded on every open subset of $\Jc_f$.
We use the convention that a function taking value $\infty$ at some point is
 automatically unbounded on every set containing this point.
Note that the conjecture holds trivially if $\Jc_f=\0$.
Also, it may be important to observe that $\lev$ is always locally unbounded on all of $\Fl(\ol\Delta)$ for trivial reasons,
 however, whether the same property holds for its restriction to $\Jc_f$ is not so clear.
Assuming the conjecture, it is straightforward to prove the Main Theorem.
Indeed, take $f\in\PET(\Delta)$ such that $f$ has no aperiodic points but $\Jc_f\ne\0$.
Being a compact Hausdorff space, $\Jc_f$ has the \emph{Baire property}, namely, it cannot be
 represented as a countable union of closed nowhere dense subsets.
Since $\lev$ is locally unbounded and upper semi-continuous, the sets $\{\lev\le n\}$ are
 nowhere dense and closed in $\Jc_f$; it follows from the Baire property that they cannot
 cover all of $\Jc_f$.
Hence, there is a point $q\in\Jc_f$ with $\lev(q)=\infty$, its natural projection to $\ol\Delta$
 is an aperiodic point, and we are done.
\end{rem}

Unfortunately, since we do not have a proof of the conjecture at the moment,
 we have to take a roundabout path to proving Theorem \ref{t:aper}.
Namely, we locate a certain invariant subset of $\Jc_f$, on which $\lev$ is locally unbounded.
Removing this subset reduces the complexity in some sense, and we proceed by transfinite induction.

\subsection{Proper convergence and rank}
\label{ss:rk}
For the next definition, we need ordinal numbers and the principle of transfinite induction.
Recall that every well-ordered set has its \emph{ordinal number}, or the \emph{ordinal};
 two well-ordered sets give rise to the same ordinal if and only if they are order isomorphic.
Ordinal numbers of finite sets are naturally identified with nonnegative integers.
Suppose that $\al$, $\be$ are the ordinal numbers of well-ordered sets $A$, $B$, respectively.
Write $\al<\be$ if $A$ is isomorphic to $\{b\in B\mid b<b_0\}$, for some $b_0\in B$;
 this is a total ordering which makes any set of ordinals into a well-ordered set.
If $\al$ is an ordinal number of a well-ordered set $A$, then $\al+1$ is defined as the smallest
 ordinal that is bigger than $\al$.
A \emph{limit ordinal} is a one \emph{not} of the form $\al+1$.
Define a \emph{transfinite sequence} as a family of objects $x_\al$ indexed by
 ordinals $\al$ that are less than a given ordinal $\al_0$.
In order to \emph{define a sequence $(x_\al)$ by transfinite induction}, it suffices
 to define $x_0$ and, for every $\al<\al_0$, to define $x_\al$ in terms of the
 already defined objects $x_\be$ with $\be<\al$.
Similarly, proving a transfinite sequence of claims $\mathfrak{P}_\al$ by transfinite
 induction involved proving $\mathfrak{P}_0$ and, for $\al<\al_0$, deducing $\mathfrak{P}_\al$
 form the claims $\mathfrak{P}_\be$ with $\be<\al$.

Contrary to a popular belief, transfinite induction does not rely on the Axiom of Choice (AC),
 even though the AC is sometimes needed to construct the set of ordinals, over which the induction is performed.
We use transfinite induction below.
However, it can be shown that considering only countable ordinals and
 using only a countable version of the AC is enough for our purposes;
 this is the version of the AC used implicitly many times in any undergraduate Calculus textbook.
See also \cite{Sim09} for descriptions of second order arithmetic subsystems, \emph{without the AC},
 that allow for Cantor--Bendixon type transfinite induction arguments.

\begin{dfn}[Proper convergence and rank]
\label{d:rank}
Consider a compact first countable Hausdorff space $X$ and a lower semi-continuous function $\xi:X\to\R_{\ge 0}\cup\{+\infty\}$.
Say that a sequence of points $x_n\in X$ \emph{converges properly} to $x\in X$ if $x_n\to x$ while $x_n\ne x$ and $\xi(x_n)\to\infty$.
Given any subset $A\subset X$, define \emph{proper limit points of $A$} as limits of properly converging sequences.
Now, define a transfinite sequence of subsets $X_\al(\xi)\subset X$ as follows.
Set $X_1(\xi)=X$.
Also, for every ordinal $\al>0$, define $X_{\al+1}(\xi)$ as the set of all proper limit points of $X_\al(\xi)$;
 for a limit ordinal $\al$, define $X_\al(\xi)$ as the intersection of all $X_\be(\xi)$ with $\be <\al$.
The transfinite sequence $X_\al(\xi)$ stabilizes at some $\al_1$, that is, $X_{\al_1}(\xi)=X_{\al_1+1}(\xi)$ while
 $X_\be(\xi)\ne X_{\al_1}(\xi)$ for $\be<\al_1$.
Define the \emph{rank of $(X,\xi)$} as this ordinal number $\al_1$.
\end{dfn}

Note that every (sufficiently long) transfinite sequence of elements of a given set contains repetitions,
 simply because there are more ordinals than elements of any given set.
For a monotone sequence like $(X_\alpha(\xi))$, repetition means stabilization.
In the case of outer billiards, it is also true that stabilization happens at a finite or countable ordinal (but we do not need this),
 and it may well be the case that $\alpha_1=1$ (but we do not know).

Replacing ``proper limit points'' with just ``limit points'' in Definition \ref{d:rank} yields
 the well-known concept of the \emph{Cantor--Bendixon rank} of $X$;
 the version of $X_{\al_1}$, called the \emph{derivative of $X$ of order $\al_1$},
 is then the \emph{perfect part} of $X$, i.e., the biggest by inclusion perfect subspace of $X$
 (recall that a compact metric space is said to be \emph{perfect} if it has no isolated points).

Now let $f$, again, be a polygon exchange transformation of a bounded open polygon $\Delta$.
Define the \emph{rank $\rk(f)$ of $f$} as the rank of $(\Jc_f,\lev)$. 
Our proof of Theorem \ref{t:aper} will go by transfinite induction on $\rk(f)$.

\begin{lem}
\label{l:Yal}
Let $X$ be a compact first countable Hausdorff space, $\xi:X\to\Z_{\ge 0}$ be a lower semi-continuous function on $X$,
 and $\al_1$ be the rank of $(X,\xi)$.
Then $X_{\al_1}(\xi)=\0$, and $\al_1$ has the form $\al_0+1$ for some ordinal $\al_0$.
Moreover, $\xi$ is bounded on $X_{\al_0}(\xi)$.
\end{lem}

Note an important assumption that $\xi$ takes finite values only!

\begin{proof}
By way of contradiction, assume that $Y:=X_{\al_1}(\xi)\ne\0$.
For every integer $m\ge 0$, consider the closed set $Y_{\le m}:=\{y\in Y\mid \xi(y)\le m\}$.
Since $X_{\al_1+1}(\xi)=Y$, the set $Y_{\le m}$ is nowhere dense in $Y$.
On the other hand, $Y$ is the union of $Y_{\le m}$ over all $m\in\Z_{\ge 0}$.
By Baire's theorem (see, e.g. Section 7.3 of \cite{KF75}):
 a countable union of closed nowhere dense subsets of $Y$ has empty interior, in particular, cannot be $Y$.
The contradiction thus obtained shows that $Y=\0$.

If $\al_1$ has no immediate predecessor, then it is a limit ordinal,
 in which case $X_{\al_1}(\xi)$ can be obtained as the intersection of $X_\be(\xi)$ for all $\be<\al$.
However, the intersection of a family of nonempty compacta that is totally ordered
 by inclusion is nonempty; a contradiction with $X_{\al_1}(\xi)=\0$.

If $\xi$ is unbounded on $X_{\al_0}(\xi)$, then there is a sequence $x_n\in X_{\al_0}(\xi)$ with $\lev(x_n)\to \infty$.
Passing to a converging subsequence, assume that $x_n\to x$.
Then $x\in X_{\al_1}(\xi)$, a contradiction with $X_{\al_1}(\xi)=\0$.
\end{proof}

Call $X_{\al_0}(\xi)$ the \emph{core of $(X,\xi)$}.
If $f$ is a polygon exchange transformation of a bounded open polygon $\Delta$,
 write $\Kc_0=\Kc_0(f)$ for the core of $(\Jc_f,\lev)$.

One can now add dynamics to the construction of the core.
Suppose that $h:X\to X$ is a homeomorphism and that $\xi$ is \emph{quasi-invariant} under $h$,
 that is, $\xi\circ h(x_n)\to\infty$ if and only if $\xi(x_n)\to \infty$.
Invariant functions ($\xi\circ h=\xi$) are quasi-invariant, but there are of course
 many other quasi-invariant functions.
For example, the function $\lev$ on $\Xc_f$ is quasi-invariant under $f_*$
 since $\lev(q)-1\le \lev(f_*(q))\le \lev(q)+1$.

\begin{lem}
\label{l:core-inv}
Under the assumptions just made, the core of $(X,\xi)$ is $h$-invariant.
\end{lem}

\begin{proof}
Let us show by transfinite induction that $X_\al(\xi)$ is $h$-invariant, for every ordinal $\al$.
For $\al=0$, one has $X_0(\xi)=X$ and $h(X)\subset X$, which proves the induction base.
Suppose now that $\al>0$ and assume that $\al=\be+1$.
By induction $X_\be(\xi)$ is $h$-invariant.
Also, $X_\al(\xi)$ is defined as the set of proper limit points of $X_\be(\xi)$,
 so that $x\in X_\al(\xi)$ means $x=\lim x_n$, where $x_n\in X_\be(\xi)$ are all distinct from $x$
 and $\xi(x_n)\to\infty$.
It follows that, firstly, $h(x)=\lim h(x_n)$, secondly, $h(x)\ne h(x_n)\in X_\be(\xi)$, and, thirdly, $\xi(h(x_n))\to\infty$;
 hence $h(x)\in X_\al(\xi)$.
Here, the first claim is the continuity of $h$, the second claim is the injectivity of $h$ and the induction hypothesis,
 and the third claim is the quasi-invariance of $\xi$ under $h$.
One also needs to consider the case when $\al$ is a limit ordinal;
 in this case, $X_\al(\xi)$ is defined as the intersection of $X_\be(\xi)$ for all $\be<\al$.
The intersection if $h$-invariant sets is $h$-invariant.
\end{proof}

\subsection{Centers}
\label{ss:cent}
Start by considering an arbitrary homeomorphism $h:X\to X$
 of a nonempty compact Hausdorff space $X$ with no periodic points.
Recall that a point $x\in X$ is \emph{non-wandering} if, for every neighborhood $O$ of $x$ in $X$,
 there is a positive integer $n$ with $O\cap h^n(O)\ne\0$.
Let $\mathrm{NW}(X,h)$ denote the set of all non-wandering points of $X$.
It is not necessarily the case that, for the restriction of $h$ to $\mathrm{NW}(X,h)$, all points are non-wandering
 (because the neighborhoods are now considered in a smaller space).
However, iterating the operation $X\mapsto \mathrm{NW}(X,h)$ by transfinite induction,
 one obtains the following known result, cf. \cite[Section 3.3]{KH95}.

\begin{thm}
\label{t:wrp}
Suppose that $X$ is a nonempty compact Hausdorff space, and $h:X\to X$ is a homeomorphism without periodic points.
Then $X$ contains a nonempty perfect invariant subset $Y$ such that $\mathrm{NW}(Y,h)=Y$
 and any compact invariant subset $Z\subset X$ with $\mathrm{NW}(Z,h)=Z$ is contained in $Y$.
\end{thm}

The subset $Y\subset X$ from Theorem \ref{t:wrp} is clearly unique;
 it is called the \emph{center} of $(X,h)$ and is denoted by $X_\cent$.
As follows from Theorem \ref{t:wrp}, the center of $h$ is the maximal by
 inclusion compact invariant subset $Y$ of $X$ with the property $\mathrm{NW}(Y,h)=Y$;
 this implies the following corollary.

\begin{cor}
\label{c:wrp-inc}
Suppose that $h:X\to X$ is a homeomorphism of a compact metric space $X$ without periodic points.
For any compact $h$-invariant subset $Y\subset X$, one has $Y_\cent\subset X_\cent$.
\end{cor}

By a theorem of Cantor, every nonempty perfect metric space is uncountable.
It follows from Theorem \ref{t:wrp} that $X$ is uncountable;
 in particular, $\Xc_f$ is uncountable if nonempty.

Informally speaking, the main engine that promotes aperiodicity to the existence of aperiodic orbits
 is that, if the desired conclusion fails, we can use the nonempty space $\mathcal{J}_f$ to isolate some bad behavior
 and encapsulate it in a sub-system with zero dynamical invariants.
This lets us peel back a layer of the dynamics, so to speak, and get a
 simpler system with similar problem.
Theorem \ref{t:rem-core} stated below is the technical result that makes this possible.

\begin{thm}
\label{t:rem-core}
Let $f$ be an unbranched polygon exchange transformation of a bounded open polygon $\Delta$
 with no aperiodic points.
Set $\al_0+1:=\rk(f)$, let $\Kc_0:=\Jc_{f,\al_0}(\lev)$ be the core of $(\Jc_f,\lev)$,
 and define $\Kc$ as the center of $(\Kc_0,f_*|_{\Kc_0})$.
There is a polygon (possibly disconnected) $\Delta_0\subset\Delta$ such that $\Kc\subset \Fl(\ol\Delta_0)$ and
\begin{enumerate}
  \item the restriction $f|_{\Delta_0}$ represents an element of $\mathrm{PET}(\Delta_0)$;
  \item all dynamic invariants $\Inv_L(f|_{\Delta_0})$ vanish.
\end{enumerate}
\end{thm}

Recall that $\Kc_0$ is $f_*$-invariant, by Lemma \ref{l:core-inv},
 so that $(\Kc_0,f_*)$ does indeed make sense as a dynamical system.
We first prove Theorem \ref{t:aper} assuming Theorem \ref{t:rem-core}, and then
 proceed with proving the latter.

\begin{proof}[Proof of Theorem \ref{t:aper} assuming Theorem \ref{t:rem-core}]
Let $\al_1=\al_0+1$ be the rank of $f$.
Suppose first that $\al_1=1$.
This means that $\Jc_f=\0$, hence $f$ is periodic by Proposition \ref{p:Jc-empty}, a contradiction;
 the implication of Theorem \ref{t:aper} is therefore true since the assumptions are incompatible.
By transfinite induction on the rank, we may now assume that Theorem \ref{t:aper}
 holds for all unbranched polygon exchange transformations of bounded polygons,
 whose ranks are less than $\al_1$.
Let $\Kc$ and $\Delta_0$ be as in Theorem \ref{t:rem-core}.
Note that the restriction $f'$ of $f$ to $\Delta\sm\ol\Delta_0$ is also unbranched.
Since all dynamic Hadwiger invariants of $f|_{\Delta_0}$ vanish,
 the map $f'$ has the same dynamic Hadwiger invariants as $f$.
It suffices to show that $\rk(f')<\al_1$.

Set $X:=\Jc_f$ and $X':=\Jc_{f'}$.
Also, set $\xi:=\lev$.
Observe that
 $\Fl(\ol\Delta\sm\Delta_0)\subset\Fl(\ol\Delta)$, hence $X'\subset X$.
Let $X'_\al(\xi)$ and $X_\al(\xi)$, for an ordinal $\al$, be as in Definition \ref{d:rank}.
As follows immediately from the definition, $X'_\al(\xi)\subset X_\al(\xi)$ for all $\al$;
 in particular, $X'_{\al_0}(\xi)\subset X_{\al_0}(\xi)$.
Re-denote $X_{\al_0}(\xi)$ as $Y$ and $X'_{\al_0}(\xi)$ as $Y'$;
 both compacta $Y$ and $Y'$ are invariant under $f_*$.
Since $Y'_{\cent}\subset\Kc$ and $\Kc\subset\Fl(\ol\Delta_0)$ while $X'\cap\Fl(\ol\Delta_0)=\0$, we have $Y'_{\cent}=\0$.
Hence also $Y'=\0$ by Theorem \ref{t:wrp}, and then $\rk(f')<\al_1$, as desired.
\end{proof}

Since Theorem \ref{t:aper} implies the Main Theorem, and Theorem \ref{t:rem-core}
 implies Theorem \ref{t:aper}, it now remains to prove Theorem \ref{t:rem-core}.

\section{Reduction to dimension one}
\label{s:red1D}
This section proves Theorem \ref{t:rem-core}, which implies the Main Theorem.
Our standing assumptions in Section \ref{s:red1D} are as follows.
Let $f$ be an unbranched polygon exchange transformation of a bounded polygon $\Delta$
 with no aperiodic points.
As in the statement of Theorem \ref{t:rem-core}, let $\Kc_0$ be the core of $(\Jc_f,\lev)$
 and define $\Kc$ as the center of $\Kc_0$.
By Lemma \ref{l:Yal}, the level function $\lev$ is bounded on $\Kc_0$, hence also on $\Kc$.
This means that the impressions of all nests from $\Kc$ are contained in the same finite graph
$$
\Gamma_{f,k_*}:=\ol\Delta\sm (\dom(f^{k_*})\cap\dom(f^{-k_*})),
$$
 for some $k_*\ge 0$.
One can take $k_*:=\max(\lev|_{\Kc})$.

Informally, there are three main steps involved in Theorem \ref{t:rem-core}.
\begin{enumerate}
  \item Extend $f$ to some boundary points and obtain a restriction $\varphi_L$
   of the thus obtained map which is an interval exchange transformation acting on
   some finite union $K_L$ of intervals (Sections \ref{ss:decKc} and \ref{ss:fL}).
  \item Secondly, thicken up the map $\varphi_L$ and make it into a polygon exchange transformation
   defined on certain very thin trapezoids
  that are based on the segments of $K_L$ and defined in terms of elementary sectors
  (Definition \ref{d:trains} and Lemma \ref{l:trains}).
  \item Thirdly, show that the dynamical invariant of this new polygon
  exchange transformation vanishes (Theorem \ref{t:hor-all0}),
  thereby establishing Theorem \ref{t:rem-core}.
\end{enumerate}

\subsection{A decomposition of $\Kc$}
\label{ss:decKc}
Two co-oriented lines $L$, $L'$ are said to be \emph{parallel} if
 the positive half-plane with respect to $L$ is contained in that for $L'$, or the other way around.
Parallelism thus defined is an equivalence relation.
Define $\Fl_L(\ol\Delta)$ as the set of all flags $q\in\Fl(\ol\Delta)$ of the form $q=(a,R,H)$,
 where the line $L_R$ containing $R$ and co-oriented according to $H$ is parallel to $L$.
Set $\Kc_L:=\Kc\cap\Fl_L(\ol\Delta)$; this is clearly a compact $f_*$-invariant subset of $\Kc$.
Also, let $K_L$ (resp., $K$) be the natural projection of $\Kc_L$ (resp., $\Kc$) to $\ol\Delta$.
Let $\Lc_f$ be the finite set of co-oriented lines through the origin that are parallel
 to nondegenerate straight line segments contained in $\Gamma_f$;
 lines of $\Lc_f$ come in pairs of oppositely co-oriented but otherwise coincident lines.

\begin{lem}
\label{l:dec-Kc}
The set $K$ is the union of $K_L$, where $L\in\Lc_f$.
On the other hand, if $L_1\ne L_2$, then $\Kc_{L_1}\cap\Kc_{L_2}=\0$.
\end{lem}

Note: we do \emph{not} claim $K_{L_1}\cap K_{L_2}=\0$ for $L_1\ne L_2$.

\begin{proof}
The union of all lines from $\Lc_f$ divides the unit tangent space $\UT_0$ at the origin into several open arcs.
Let $A$ be one of these arcs.
Define $\Fl(a,\ol A)$ as the set of all flags $(a,R,H)$, where the direction of $R$ is in $\ol A$
 and, if $R$ corresponds to an endpoint of $A$, then $H$ is on the same side of $R$ as $A$.
Call $\Fl(a,\ol A)$ the \emph{$A$-arc of $\Fl_a(\ol\Delta)$}.
It is easy to see that $f_*(\Fl(a,\ol A))=\Fl(b,\ol A)$ for the same point $b$
 that depends only on $a$ and $\ol A$, that is, $A$-arcs are mapped onto $A$-arcs.
Within any given $A$-arc, the dynamics of $f_*$ is thus the same.
In particular, the Julia set $\Jc_f$ and the set $\Kc$ are both unions of $A$-arcs.

Choose any $q\in\Kc$, and let $a$ be the natural projection of $q$ to $\ol\Delta$.
Then $q$ lies in an $A$-arc, for some $A$.
Also, the endpoints of $A$ give rise to two flags $q_1$, $q_2$ and, respectively,
 to lines $L_1$, $L_2\in\Lc_f$ co-oriented so that $A$ lies on their positive sides.
If $q_i=(a,R_i,H_i)$ for $i=1$, $2$, then $R_i$ lies in a parallel translate of $L_i$,
 and the corresponding translate of $H_i$ contains $A$.
Note that all flags from $\Fl(a,\ol A)$ map to the same point $a$ under the natural projection.
Since $q\in\Kc$, and the entire $A$-arc has the same dynamics, $q_i\in\Kc_{L_i}$ for $i=1$, $2$.
It follows that $K$ is the union of $K_{L}$ for $L\in\Lc_f$.
The last claim of the lemma is immediate from the definitions.
\end{proof}

\subsection{An extension of $f$}
\label{ss:fL}
Fix a co-oriented line $L\in\Lc_f$.
Denote by $f_L$ a special extension of $f$ defined as follows.
On $\dom(f)$, set $f_L=f$.
Given an maximal by inclusion straight line interval $e\subset\ol\Delta\sm\dom(f)$
 parallel to $L$ and not including the endpoints,
 define $f_L$ on $e$ as the limit of $f$ from the positive side of $e$,
 where the positive side is determined by the co-orientation of $L$.
Thus, $f_L$ is defined on $\dom(f)$ and, additionally, on some complementary intervals parallel to $L$.

\begin{lem}
  \label{l:fLonJ*}
The map $f_L$ is defined at all points of $K_L$ except at most finitely many,
 and $K_{L}$ is $f_L$-invariant in the sense that $f_L(K_L\cap\dom(f_L))\subset K_L$.
Points of $K_L$ are non-wandering and non-periodic under $f_L$.
\end{lem}

\begin{proof}
The first claim of the lemma is immediate from the definitions.
Points of $K_L$ are non-wandering and non-periodic under $f_L$
 due to the corresponding properties of $\Kc_L$ and the
 continuity of the natural projection.
\end{proof}

Assume $K_L\ne\0$ in what follows.
Observe that $K_L$ is a perfect set, since points of $K_L$ are non-wandering and non-periodic by Lemma \ref{l:fLonJ*}.
We think of $L$ as being a horizontal line, and will refer to the positive side of $L$ as the \emph{upper side}.
Define $I_L$ as the union of all maximal by inclusion straight line intervals $e$
 such that, firstly, $R\cap e$ is one dimensional, for some $(a,R,H)\in\Kc_L$, and, secondly,
 $f_L$ is defined everywhere on $e$.
Since the level function is bounded on $\Kc_L$, the set $I_L$ consists of finitely many intervals.
Note that $f_L$ acts on $I_L$ as a translation scissors congruence, however,
 it does not have to map $I_L$ to itself.
Extend the restriction of $f_L$ to $I_L$ somehow to an interval exchange transformation $\varphi_L$ of $I_L$,
 e.g., set $\varphi_L=f_L$ on $f_L^{-1}(I_L)\cap I_L$, and define $\varphi_L$ on the complement
 $I_L\sm f^{-1}(I_L)$ as any translation scissors congruence between $I_L\sm f^{-1}(I_L)$ and $I_L\sm f(I_L)$
 (such a scissors congruence exists since the two sets have the same total length).
We need some general properties of interval exchange maps.

\begin{thm}[\cite{KH95,NPT13}]
\label{t:KH}
Let $\varphi$ be an interval exchange transformation of an interval (or a finite union of intervals) $I$.
There exists a finite set of open $\varphi$-invariant subsets $I_1$, $\dots$, $I_m\subset I$ with the following properties:
\begin{itemize}
  \item Each $I_i$ is a finite union of intervals;
  \item Each $I_i$ is either a \emph{periodic component}, i.e., $I_i$ is a cycle of intervals
  not properly contained in a cycle of bigger intervals,
   or a \emph{minimal component}, i.e., all infinite forward or backward orbits in $I_i$ are dense.
\end{itemize}
Also, all endpoints of all $I_i$ are boundary points of $\varphi$, and
 $m$ is bounded above by the number of intervals being exchanged by $\varphi$.
\end{thm}

As always, $\varphi$-invariance of a subset $A\subset I$ is understood as $\varphi(A\cap\dom(\varphi))\subset A$,
 which is reasonable since $A\sm\dom(\varphi)$ is finite.
See Theorem 14.5.13 in \cite{KH95} and Theorem A of \cite{NPT13} for a sharper upper bound on $m$.
For our purposes, results of \cite{KH95} are enough.
By Lemma \ref{l:fLonJ*} and Theorem \ref{t:KH}, the set $K_L$ consists of
 one or several minimal components of $\varphi_L$.
In particular, it is a finite union of intervals on which $f_L$ acts as an interval exchange transformation.

\subsection{Forward and backward trains}
Consider a component $e$ of $K_L\cap\dom(f_L)$; this is a straight line interval
 (not including the endpoints) parallel to $L$.
All flags $(a,R,H)$ with $a\in e$ and $H$ being the upper side of $e$ are mapped
 forward by the same translation $g$.
Therefore, there exists an open trapezoid $P$ with parallel horizontal edges and with the property that
 the bottom edge of $P$ is $e$ and that $f=g$ on $P$.
Making $P$ smaller, we can arrange that the height of $P$ is a given sufficiently small
 number $\eps>0$, and, for every bottom vertex $b$ of $P$, no line through $b$ parallel to
 a line from $\Lc_f$ enters inside $P$.
Such a trapezoid is called a \emph{forward trapezoid}.
\emph{Backward trapezoids} are defined in the same way, replacing $f_L$ with $f^{-1}_L$ but with the same $K_L$.

\begin{dfn}[Forward and backward trains]
\label{d:trains}
Let $\eps>0$ be a sufficiently small real number, in particular, smaller than the
 distance between any two distinct lines containing components of $K_L$.
The union of the forward trapezoids of height $\eps$ is called the \emph{forward $\eps$-train of $K_{L}$},
 or just the \emph{forward train} of $K_{L}$ if $\eps$ is fixed.
\emph{Backward trains} are defined similarly.
\end{dfn}

See Fig. \ref{fig:hor-PET} (top) for a hypothetical situation
 where a forward train does not match the corresponding backward train
 (this situation is \emph{not} unbranched).

\begin{lem}
\label{l:trains}
The map $f_L$ defines a translation scissors congruence from the forward train of $K_L$ to the backward train.
Since $f$ is unbranched, both forward and the backward trains are subdivisions of the same polygon;
 thus, $f_L$ defines a polygon exchange transformation of this polygon.
\end{lem}

\begin{proof}
Note that the $\eps$-trains of $K_L$, both forward and backward, are contained in
 an arbitrarily small neighborhood of $K_L$, provided that $\eps$ is small enough
 for the chosen neighborhood.
The set theoretic difference between the backward train and the forward train,
 with the same $\eps$, consists of several triangles (possibly degenerate, i.e., line segments)
 cut out of the backward train by straight line intervals where $f$ is not defined.
As such a triangle must have a vertex on the horizontal base of some backward trapezoid,
 the triangle must degenerate in the unbranched case.
Similarly, the forward train minus the backward train is a union of finitely many
 triangles cut out of the forward train by the line intervals where $f^{-1}$ is not defined;
 see Fig. \ref{fig:hor-PET}, top.

If $f$ is unbranched, then there can be no non-degenerate triangles as above.
Therefore, the forward and the backward trains are subdivisions of the same polygon,
 and $f$ acts on this polygon as a polygon exchange transformation.
Moreover, the restriction of this polygon exchange transformation to every component of
 its domain is a parallel translation by a vertor parallel to $L$, that is, by a horizontal vector.
\end{proof}

\subsection{Horizontal polygon exchange transformations}
\label{ss:h-pet}
Fix a co-oriented line $L$ and, as before, refer to its direction as the horizontal direction,
 and to its positive side as the upper side.
Suppose that a polygon exchange transformation $f_\Theta$ of a bounded open polygon $\Theta$
 translates all components of $\dom(f_\Theta)$ by vectors parallel to $L$.
In this case, $f_\Theta$ is called a \emph{horizontal polygon exchange transformation} with respect to $L$.
See Fig. \ref{fig:hor-PET} for a non-trivial example.
Our standing assumption in Section \ref{ss:h-pet} is that $f_\Theta$ is a horizontal
 polygon exchange transformation with respect to $L$.
There are several special properties of $f_\Theta$.
Any horizontal line $L_t$ parallel to $L$ intersects $\Theta$ over a finite union of horizontal intervals.
Note that $f_\Theta$ induces an interval exchange transformation of $\Theta\cap L_t$.
Here, the parameter $t$ can be understood, e.g., as a height over $L=L_0$,
 and we write $\varphi_t$ for the interval exchange transformation of $\Theta\cap L_t$ induced by $f_\Theta$.

\begin{figure}
  \centering
  \includegraphics[width=\textwidth]{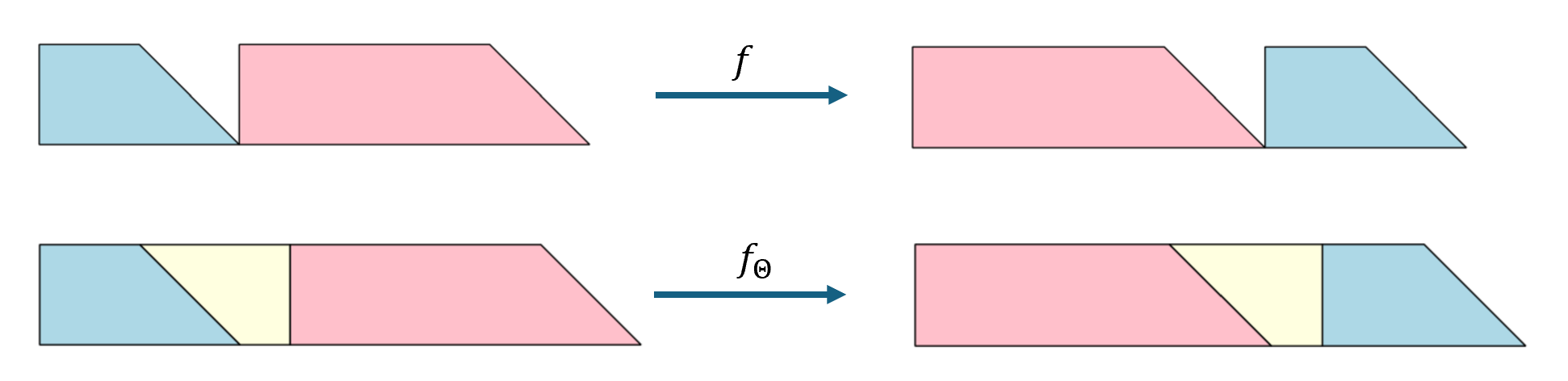}
  \caption{\small
  Top: a hypothetical picture where a forward train and a backward train
  do not match; such behavior is impossible in an unbranched case.
  Bottom: a non-trivial horizontal polygon exchange transformation.
  }\label{fig:hor-PET}
\end{figure}

\begin{thm}
\label{t:hor-all0}
All dynamic Hadwiger invariants of a horizontal polygon exchange transformation $f_\Theta$ are zero.
\end{thm}

Postponing the proof of Theorem \ref{t:hor-all0} until Section \ref{ss:dyn-hor},
 we now deduce Theorem \ref{t:rem-core},
 hence also the Main Theorem, from Theorem \ref{t:hor-all0}.

\begin{proof}[Proof of Theorem \ref{t:rem-core} using Theorem \ref{t:hor-all0}]
We use the notation introduced in the statement of Theorem \ref{t:rem-core} and in Section \ref{ss:decKc}.
Each $L\in\Lc_f$ gives rise to a perfect compactum $K_L\subset \ol\Delta$, which is in fact a finite union of compact intervals.
Moreover, $K$ is the union of all the $K_L$, while the corresponding sets $\Kc_L\subset\Jc_f$ are pairwise disjoint, see Lemma \ref{l:dec-Kc}.
Choose a sufficiently small $\eps>0$, and note that the $\eps$-trains of $K_L$ for various $L$ are pairwise disjoint.
Indeed, directions of translations corresponding to geometrically distinct $L$s are different.
Set $\Delta_0$ to be the union of all these trains.

Having defined a polygon $\Delta_0$, we need to verify that $\Kc\subset\Fl(\ol\Delta_0)$
 as well as properties (1) -- (2) from Theorem \ref{t:rem-core}.
Property (2) follows from Theorem \ref{t:hor-all0}, property (1) is contained in Lemma \ref{l:trains},
 and, finally, the inclusion $\Kc\subset\Fl(\ol\Delta_0)$ follows from the fact that any
 flag $q\in\Kc$ belongs to some forward trapezoid.
\end{proof}

\subsection{Dynamic invariants vanish for horizontal exchanges}
\label{ss:dyn-hor}
Below, we prove Theorem \ref{t:hor-all0}.
Start with an important special case.

\begin{thm}
\label{t:hor-0}
If $f_\Theta$ is a horizontal polygon exchange transformation with respect to $L$, then $\Inv_L(f_\Theta)=0$.
\end{thm}

Before giving a proof of Theorem \ref{t:hor-0}, we make several preparatory remarks.
There are finitely many lines $L_{t_0}$, $\dots$, $L_{t_m}$ such that the vertices
 of all components of $\dom(f_\Theta)$ are contained in the union of these lines;
 recall that $t_0$, $\dots$, $t_m$ are the corresponding heights.
The part $\Theta_i$ of $\Theta$ between $L_{t_i}$ and $L_{t_{i+1}}$
 is a union of trapezoids with horizontal edges in $L_{t_i}$ and $L_{t_{i+1}}$.
Some of these edges may  degenerate to points, letting the corresponding
 trepezoids degenerate to triangles.
Call the thus obtained trapezoids between $L_{t_i}$ and $L_{t_{i+1}}$ the
 \emph{trapezoids on the $i$-th floor}.
Observe that $f_\Theta$ acts as a horizontal polygon exchange transformation on $\Theta_i$.
Knowing that the dynamic invariant $\Inv_L$ is additive with respect to
 polygonal subdivisions of the domain, it suffices to prove Theorem \ref{t:hor-0}
 for the union of all trapezoids on a given floor.
Moreover, if we consider the first return map to any given trapezoid and
 then subdivide it further into floors if necessary, then we reduce Theorem \ref{t:hor-0}
 to the case when $\Theta$ is just an \emph{$L$-trapezoid}, i.e., a trapezoid with top and bottom edges,
 possibly degenerate, parallel to $L$.
We therefore assume now that $\Theta$ is an $L$-trapezoid.

Under the assumptions made, the dynamics of $f_\Theta$ reduces to that of interval exchange maps.
Rescale the parameter $t$ so that the bottom and the top edges of $\Theta$ lie in $L_0$ and $L_1$, respectively.
The family of interval exchange transformations $\varphi_t$, $t\in [0;1]$, is linear
 in the following sense: $\varphi_t$ translates the intervals $I_1(t)$, $\dots$, $I_m(t)$
 by the same (one-dimensional) translation vectors $\tau_1$, $\dots$, $\tau_m$ (independent of $t$),
 and the length $\la_i(t)$ of $I_i(t)$ depends linearly on $t\in [0;1]$ for all $i=1$, $\dots$, $m$.
It is possible that $\la_i(0)=0$ or $\la_i(1)=0$ but $\la_i(t)>0$ for all $t\in (0;1)$.
A family $\varphi_t$ of interval exchange transformations as just described will be
 referred to as a \emph{linear one parameter family of interval exchange transformations}.
Next, we recall some basic data associated with interval exchange transformations.

\begin{dfn}[Data associated with IETs]
Consider an interval exchange transformation $\varphi$ of some interval $I$.
Write $I_1$, $\dots$, $I_m$ for the components of $\dom(\varphi)$.
Assume that they are numbered from left to right.
Let $\pi(i)$ be defined so that $\varphi(I_i)$ is the $\pi(i)$-th component of $\dom(\varphi^{-1})$ from the left;
 then $\pi=\pi_\varphi$ is a permutation of $\{1,\dots,m\}$ called the \emph{permutation of $\varphi$}.
Define the \emph{vector of lengths} of $\varphi$ as the vector $\la\in\R^m$ whose components
 are the lengths $\la_1$, $\dots$, $\la_m$ of subintervals $I_1$, $\dots$, $I_m$.
The \emph{vector of translations} of $\varphi$ is by definition the vector $\tau\in\R^m$
 formed by the numbers $\tau_1$, $\dots$, $\tau_m$ such that $\varphi|_{I_i}$ acts
 as the translation $x\mapsto x+\tau_i$.
\end{dfn}

Recall also that the vector of lengths and a permutation $\pi$ of $\{1,\dots,m\}$ determine the vector of translations
 as follows: $\tau_k=\sum_{\pi(i)<\pi(k)} \la_i-\sum_{j<k}\la_j$.
The latter linear relation can be rewritten in the matrix form as $\tau=\Omega\la$
 and in coordinates as $\tau_i=\sum_{j=1}^{m} \Omega_{ij}\la_j$, where
$$
\Omega_{ij}=\begin{cases}1,& \hbox{if } \pi(i)>\pi(j) \hbox{ and } i<j,\\
-1,& \hbox{if } \pi(i)<\pi(j)\hbox{ and } i>j,\\
0,& \hbox{in all other cases}.
\end{cases}
$$
As is seen directly from the formula just displayed, the matrix $\Omega$ is skew symmetric.

\begin{lem}
  \label{l:SAF-lopfiet}
Let $(\varphi_t)$ be a linear one parameter family of interval exchange transformations.
The SAF invariant of $\varphi_t$ is independent of $t$.
\end{lem}

\begin{proof}
Write $\la_i(t)$ as $\la_i(0)+t\la'_i(0)$ and substitute this expression into the formula
 $\sum \la_i(t)\otimes\tau_i$ for the invariant $\Inv(\varphi_t)$.
Recall that $\R\otimes\R$ can be viewed as a vector space over $\R$, with
 left multiplication by real scalars satisfying the identity $(ta)\otimes b=t(a\otimes b)$.
It follows that $\Inv(\varphi_t)$ can be written as the sum of two terms:
 the first term is independent of $t$ and is equal to $\Inv(\varphi_0)$,
 whereas the second term is $t$ times the sum of $\la'_i(0)\otimes \tau_i$
 over $i$ ranging from 1 to $m$.
Clearly, it suffices to show that this second term, equal to $\frac{d}{dt}\Inv(\varphi_t)|_{t=0}$, vanishes
 (the derivative makes sense since $\varphi_t$ takes values in a two-dimensional $\R$-vector subspace
 of $\R\otimes\R$).
Substituting the expression for $\tau$ as $\tau=\Omega \la(0)$, we obtain that
$$
\left.\frac{d}{dt}\Inv(\varphi_t)\right|_{t=0}=\sum_{i,j=1}^{m}\Omega_{ij}\, \la'_i(0)\otimes \la_j(0).
$$
Here, the factor $\Omega_{ij}$ has been moved to the front of the tensor product
 since $\Omega_{ij}\in\Q$.

On the other hand, differentiating each component of the matrix identity $\tau=\Omega\la(t)$  by $t$ and
 using that $\tau$ does not depend on $t$, we see that $\sum_j\Omega_{ij}\la'_j(0)=0$.
It remains to substitute the latter equality into the expression for $\frac{d}{dt}\Inv(\varphi_t)|_{t=0}$
 and use the skew symmetry of the matrix $\Omega$.
We obtain that $\Inv(\varphi_t)$ is constant with respect to $t$, as claimed.
\end{proof}

We are now ready to complete the proof of Theorem \ref{t:hor-0}.

\begin{proof}[Proof of Theorem \ref{t:hor-0}]
By the above, it suffices to assume that $\Theta$ is an $L$-trapezoid
 between $L_0$ and $L_1$, so that $\varphi_t$ for $t\in [0;1]$ form a
 linear one parameter family of interval exchange transformations.
Since all the translation vectors are horizontal, the space $\R^2$ of translations reduces to $\R^1$,
 where all $\tau_i(t)$ lie.
With this identification, the dynamic Hadwiger invariant of $f_\Theta$ identifies
 with the difference $\Inv(\varphi_0)-\Inv(\varphi_1)$ of two SAF invariants.
By Lemma \ref{l:SAF-lopfiet}, the latter difference is zero.
\end{proof}

\begin{proof}[Proof of Theorem \ref{t:hor-all0}]
Vanishing of $\Inv_L(f_\Theta)$ follows from Theorem \ref{t:hor-0}.
Let $L'$ be a line transverse to $L$; we want to show that $\Inv_{L'}(f_\Theta)=0$.
It suffices to assume that $\Theta$ is an $L$-trapezoid between $L_0$ and $L_1$.
Also, one may assume without loss of generality that the positive side of $L'$ is on the right of $L'$.
Write $\ell'$ for the length of the part of $L'$ between $L_0$ and $L_1$.

Introduce an affine coordinate system in the plane so that one of the
 coordinate axes is horizontal, i.e., parallel to $L$.
Denote by $x_i(t)$, respectively, $y_i(t)$, the horizontal coordinate of the left endpoint,
 respectively, the right endpoint, of $I_i(t)$.
Similarly, write $x'_i(t)$ and $y'_i(t)$ for the corresponding horizontal coordinates of $\varphi_t(I_i(t))$.
Let $i_1<\dots<i_k$ be all values of the index $i$, for which $\{x_i(t)\mid t\in [0;1]\}$ is
 a line interval parallel to $L'$.
Also, write $j_1<\dots<j_l$ for the values of $j$ such that
 the line interval $\{y_j(t)\mid t\in [0;1]\}$ is parallel to $L'$.
The value $\Inv_{L'}(f_\Theta)$ is by definition equal to
$$
\ell'\otimes \left(\tau_{i_1}+\dots+\tau_{i_k}-\tau_{j_1}-\dots-\tau_{j_l}\right),
$$
 and we need to show that the second factor equals zero.
Note that
\begin{align*}
&x_{i_1}(t)+\dots +x_{i_k}(t)-y_{j_1}(t)-\dots -y_{j_l}(t)=\\
=\ &x'_{i_1}(t)+\dots +x'_{i_k}(t)-y'_{j_1}(t)-\dots-y'_{j_l}(t).
\end{align*}
Indeed, both in the left-hand side and in the right-hand side,
 terms corresponding to points in the components of $\Theta\sm\dom(f_\Theta)$
 parallel to $L'$ cancel out,
 and only those terms which correspond to the side (left and right) edges of $\Theta$ possibly survive.
If the side edges of $\Theta$ are not parallel to $L'$, then there are no such terms either,
 and, if there are such terms, they are the same on both sides.
Finally, it remains to observe that $\tau_{i_r}=x'_{i_r}(t)-x_{i_r}(t)$ and
 $\tau_{j_s}=y'_{j_s}(t)-y_{j_s}(t)$, where $r=1$, $\dots$, $k$ and $s=1$, $\dots$, $l$.
\end{proof}

As we have seen earlier, Theorem \ref{t:hor-all0} implies the Main Theorem.

\section{Further perspectives}
\label{s:sum}
Our specific objective was to prove, for a regular $N$-gon $\Pi$ with $N\ne 3$, $4$, $6$, that
 the outer billiard map $f_\Pi$ has aperiodic points (the Main Theorem).
However, from a more general perspective, we also aimed at introducing tools
 for verifying non-periodicity of scissors automorphisms.
Non-periodicity can be established with the help of dynamic invariants, which are
 explicitly computable in the situation of the Main Theorem. 
A possible direction for future research is extending the Main Theorem to all non-lattice quasi-rational polygons
 (a conjecture of Gutkin and Simanyi).
Here, an additional difficulty is that it is too general for relying on closed form formulas.
General methods will be necessary for establishing the nonvanishing of dynamic invariants.

Related is a problem of better understanding what exactly nonzero dynamic invariants
 can say about dynamics, besides the lack of periodicity.
This problem is not fully resolved even in dimension one, that is,
 for interval exchange transformations.
It would be natural to first approach this problem either for special classes
 of interval exchange maps or for outer billiards on regular polygons.
Some other problems are listed below:
\begin{itemize}
\item In the setting of the Main Theorem, develop an algorithmic approach to
 finding aperiodic points.
  \item Find interesting classes of multi-dimensional polytope exchange transformations,
   for which dynamic invariants are explicitly computable.
  \item Relate polytope exchange transformations in dimension $d$
   with flows on Euclidean $(d+1)$-manifolds; describe the role of dynamic invariants
   in terms of these flows.
  \item Develop general methods for establishing (non)existence of dynamical self-similarity
   in the context of scissors automorphisms.
\end{itemize}

Finally, a significant challenge would be to extend dynamical invariants beyond
 the realm of piecewise affine dynamics, by allowing piecewise continuous maps
 that are not affine on the components of their domains
 and also allowing pieces to have more general than piecewise linear shapes.



\end{document}